\documentclass[a4paper,12pt,oneside]{article}
\usepackage[english]{babel}
\usepackage[T1]{fontenc} 
\usepackage[utf8]{inputenc}
\usepackage{amsthm}
\usepackage{bbm}
\usepackage{amsmath}
\usepackage{amssymb}  
\usepackage{indentfirst}
\usepackage{fancyhdr}
\usepackage{amsthm}
\usepackage{graphicx}
\usepackage{pdfpages}
\usepackage{esint}
\usepackage{url}
\usepackage{xcolor}
\usepackage{ulem}

\numberwithin{equation}{section}

\theoremstyle{plain} 
\newtheorem{thm}{Theorem}[section] 
 
\newtheorem{lem}[thm]{Lemma} 
\newtheorem{prop}[thm]{Proposition} 
 
\newtheorem{defn}[thm]{Definition}

\usepackage{geometry}
\allowdisplaybreaks[4]

\begin{document}
\author {Miriam Piccirillo 
\sc{}\thanks{Dipartimento di Matematica e Applicazioni "R. Caccioppoli", Università degli Studi di Napoli "Federico II", Via Cintia, 80126 Napoli, Italy. E-mail: \textit{miriam.piccirillo@unina.it}} 
  }

\title{Second-order regularity for weighted widely degenerate problems with explicit $u$-dependence}
\date{}
\maketitle

\begin{abstract}
We consider local weak solutions of widely degenerate elliptic PDEs of the type 
\begin{equation}
        \label{equazione mia}
        \mathrm{div}\Biggl(|x|^\beta(|Du|-1)^{p-1}_+\frac{Du}{|Du|}\Biggr)=\frac{|u|^{q-2}u}{|x|^\alpha} \ \ \text{ in }\Omega,
    \end{equation}
where $2\leq p<n$, $\alpha,\beta>0$ are fixed exponents, $\Omega$ is an open subset of $\mathbb{R}^n,$ that contains the origin, $n>2,$ and $( \ \cdot \ )_+$  stands for the positive part. We establish a higher differentiability result for the composition of the gradient with a suitable function that vanishes in the unit ball for the gradient, under appropriate assumptions on the datum. The novelty here with respect to previous papers on the subject is that the right-hand side explicitly depends on the solution $u$ and we have a mismatch between the weight on the left-hand side and the right-hand side.
\end{abstract}


\section{Introduction}
In this paper, we are interested in the regularity properties of the local weak solutions to strongly degenerate elliptic equation with a double spatial dependency and a singular lower-order term. Specifically, we consider equations of the form
\begin{equation}
        \label{equazione mia}
        \mathrm{div}\Biggl(|x|^\beta(|Du|-1)^{p-1}_+\frac{Du}{|Du|}\Biggr)=\frac{|u|^{q-2}u}{|x|^\alpha} \ \ \text{ in }\Omega,
    \end{equation}
where $2\leq p<n$, $\Omega$ is an open subset of $\mathbb{R}^n,$ containing the origin, $n>2,$ and $(\ \cdot \ )_+$  stands for the positive part. The operator on the left-hand side of \eqref{equazione mia} presents two different difficulties, in fact the term $|x|^\beta$ degenerates at the origin and the function of the gradient of $u$ is uniformly elliptic only outside the unit ball of the gradient, $\{|Du| > 1\}$, where it asymptotically behaves like the classical $p$-Laplace operator. Conversely, on the set $\{|Du| \le 1\}$, the elliptic structure completely collapses, as the energy density is flat. One of the main motivations for the study of equations of this type comes from the optimal transport problems with congestion effects, a connection that has been exhaustively studied in \cite{Brasco, BraCa2, BraCa, BraCaSan, BS}.\\
The study of such an equation fits into the wider class of asymptotically regular problems that have been studied starting from the pioneering paper by Chipot and Evans \cite{CE}, concerning the homogeneous, autonomous quadratic growth case. Later, still for the homogeneous autonomous case, the Lipschitz continuity of weak solutions has been established for the superquadratic growth \cite{GiaMo}, and in the subquadratic growth case \cite{LPdNV}. Since then, many contributions to the regularity theory of weak solutions of widely degenerate equations have been established. Among others, we mention the results in \cite{BoDuGiPa, C, CoFi2, EMM, FPdNV2, Gr, MP, Ru}. It is worth pointing out that no more than Lipschitz regularity can be expected for the solutions of widely degenerate problems. The substantial novelty of the present paper, which sets it apart from the aforementioned literature, lies in the simultaneous presence of two severe difficulties: the explicit dependence of the source term on the solution $u$ itself, and a structural weight mismatch between the principal part and the right-hand side.  
The dependence of the datum on $u$ is notorious for introducing major compactness issues. Even in the simplest setting of the Laplacian, this dependence constitutes a difficulty in itself, a fact that has been well known since the pioneering paper by Brezis and Nirenberg \cite{BreNir}. For what concerns the regularity we have to mention the papers  \cite{CM, CMM1, CMM2, GRu, GuJ}, where they obtained global boundedness of the weak solutions to a class of nonuniformly elliptic equations with $u$-dependent right-hand side as well as their higher differentiability, under nonstandard growth conditions but in the uniformly elliptic context. In fact, compared to \cite{CM, CMM1, CMM2, GRu, GuJ}, our problem features a spatial singularity at the origin governed by $|x|^{-\alpha}$, which acts against the degeneracy weight $|x|^\beta$. To balance this delicate weight mismatch and control the singular lower-order term, a meticulous scaling analysis is required. The main heuristic of this work is that the degeneracy of the weight forces the second-order regularity of the gradient to be sought in the sub-quadratic regime, namely for some $s < 2$. By imposing a rigorous balance between the singular exponent $\alpha$ and the growth $q$, we show 
that $H_{p/2}(Du)$ gains a weak derivative in the Sobolev space $W^{1,s}_{loc}(\Omega)$.
To state and describe our results precisely let us specify the assumptions on our data.\\
The exponents $\alpha $ and $\beta$ are assumed to satisfy 
\begin{align}
\label{alebe}\nonumber
&{\mathrm{(}\mathrm{i}\mathrm{)}} \textbf{ }
0<\beta<1 ,\\
&{\mathrm{(}\mathrm{ii}\mathrm{)}}\textbf{ }0<\alpha<n-2+\frac{2}{p}-\frac{qn}{p^*}.
\end{align}
where $q$ is an exponent such that
\begin{equation}
\label{q}
1\leq q\leq p^*-2\bigg(\frac{p+2}{n-p}\bigg) 
\end{equation}
and where, as usual, $p^*$ denotes the Sobolev conjugate exponent of $p$, i.e. $p^*=\frac{np}{n-p}$
 if $p<n.$
 \\
We remark that the case $q=1$ is trivial. Indeed, when $q=1$, the term $\vert{}u\vert{}^{q-2}u$ simplifies to $\mathrm{sgn}(u)$, which is uniformly bounded. Thus, the a priori estimates do not require any non-linear control on $u$ and follow directly from classical Sobolev embeddings.\\
Our main result is the following 
\begin{thm}
\label{teo3}
   Let $u\in W^{1,p}_{loc}(\Omega, |x|^\beta dx)$ be a weak solution of equation \eqref{equazione mia}, under assumptions \eqref{alebe} and \eqref{q}. Then there exists $1<s<2$ and a radius $R_0=R_0(n,p,q,\alpha,\beta,s)$such that $H_{p/2}(Du)\in W^{1,s}_{loc}(\Omega)$ and the following estimate holds
\begin{align*}
&\int_{B_{R/2}}|D(H_{p/2}(Du))|^s\, dx\leq C\left(\int_{{B_{R}}}|Du|^p|x|^{\beta}\, dx\right)^{\frac{s}{p}}\\
&+C\bigg(1+\int_{B_{R}}(|u|^{\frac{ps}{2}}\,+|Du|^{\frac{ps}{2}})\, dx\bigg)^{\kappa}
\end{align*} and hence
\begin{align*}
\int_{B_{R/2}}|D(H_{p/2}(Du))|^s\, dx&\leq C\bigg(1+\int_{B_{R}}(|Du|^p+|u|^p)|x|^\beta\, dx\bigg)^{\kappa}
\end{align*}
for every $B_R\subset B_{R_0}\Subset\Omega$, $\kappa=\kappa(q,p)>0$ and a constant $C=C(n, p, q, s, \alpha, \beta, R).$
\end{thm}
Let us briefly summarize the proof of the previous Theorem, describing its main points. The central part of the proof consists in establishing a suitable \textit{a priori} estimate for the derivatives of the composition between the function $H_{p/2}(\xi)$ and the gradient of the solution of the problem by testing the equation with a localized difference quotient, carefully decoupling the degenerate and singular weights via generalized Hölder inequalities and an unweighted Sobolev embedding. Next, we introduce a family of regularized problems whose solutions satisfy the assumptions of the a priori estimate and where the spatial weight is bounded away from zero and the singular source is mollified. Uniform bounds are then established, enabling us to pass to the limit as the regularization parameter vanishes.

\section{Preliminaries \label{sec:prelim}}
\selectlanguage{british}%

\subsection{Notation and essential definitions }

\selectlanguage{english}%
\noindent $\hspace*{1em}$In this paper, we denote by $C$ or
$c$ a general positive constant that may vary on different occasions.
Relevant dependencies on parameters and special constants will be
suitably emphasized using parentheses or subscripts. \foreignlanguage{british}{The
norm we use on $\mathbb{R}^{k}$, }\foreignlanguage{american}{$k\in\mathbb{N}$}\foreignlanguage{british}{,
will be the standard Euclidean one and it will be denoted by $\left|\,\cdot\,\right|$.
In particular, for the vectors $\xi,\eta\in\mathbb{R}^{k}$, we write
$\langle\xi,\eta\rangle$ for the usual inner product and $\left|\xi\right|:=\langle\xi,\xi\rangle^{\frac{1}{2}}$
for the corresponding Euclidean norm.}\\
$\hspace*{1em}$In what follows, $B_{r}(x_{0})=\left\{ x\in\mathbb{R}^{n}:\left|x-x_{0}\right|<r\right\} $
will denote the $n$-dimensional open ball centered at $x_{0}$ with
radius $r$. We shall sometimes omit the dependence on the center
when all balls occurring in a proof are concentric. Unless otherwise
stated, different balls in the same context will have the same center.\\
\foreignlanguage{british}{$\hspace*{1em}$For further needs, we now
define the auxiliary function $H_{\delta}:\mathbb{R}^{n}\rightarrow\mathbb{R}^{n}$
by 
\begin{equation}
H_{\delta}(\xi):=\begin{cases}
\begin{array}{cc}
(\vert\xi\vert-1)_{+}^{\delta}\,\frac{\xi}{\left|\xi\right|} & \,\,\mathrm{if}\,\,\,\xi\neq0,\\
0 & \,\,\mathrm{if}\,\,\,\xi=0,
\end{array}\end{cases}\label{eq:Hfun}
\end{equation}
}

\selectlanguage{british}%
\noindent where $\delta>0$ is a parameter. \\
We conclude this first part of the preliminaries by introducing the natural energy framework for our problem. Since $0 < \beta < 1$, the weighted local Sobolev space $W^{1,p}_{loc}(\Omega, |x|^\beta dx)$ is a Banach space. 

Under our structural assumptions on the exponents $\alpha$ and $q$, the integrals involved in the identity below are well-defined and finite. Specifically, the left-hand side converges by Hölder's inequality, whereas the singular right-hand side is controlled via weighted Sobolev-Hardy embeddings. This ensures that the following definition is rigorously posed.

\begin{defn}
\noindent A function $u\in W_{loc}^{1,p}(\Omega, |x|^\beta dx)$
is a \textit{local weak solution} of equation (\ref{equazione mia})
if and only if, for any test function $\varphi\in W_{0}^{1,p}$$(\Omega', |x|^\beta dx)$ with $\Omega'\Subset\Omega$,
the following integral identity holds:
\begin{equation}
\label{def}
-\int_{\Omega}\langle |x|^\beta H_{p-1}(Du),D\varphi\rangle\,\, dx\,=\,\int_{\Omega}\frac{|u|^{q-2}u}{|x|^\alpha}\varphi\,\, dx.
\end{equation}
\end{defn}

\subsection{Algebraic inequalities }

\noindent $\hspace*{1em}$In this section, we gather some relevant
algebraic inequalities that will be needed later on.
\noindent We recall the following estimate, whose proof can be
found in \cite[Chapter 12]{Lind}.
\begin{lem}
\noindent \label{lem:Lind} Let $p\in[2,\infty)$ and $k\in\mathbb{N}$.
Then, for every $\xi,\eta\in\mathbb{R}^{k}$, the following inequality
\[
\vert\xi-\eta\vert^{p}\leq\,C\left|\vert\xi\vert^{\frac{p-2}{2}}\xi-\vert\eta\vert^{\frac{p-2}{2}}\eta\right|^{2}
\]
holds for a constant $C\equiv C(p)>0$.
\end{lem}

\selectlanguage{english}%
\noindent Combining \cite[Lemma 2.2]{AceFu} with \cite[Formula (2.4)]{GiaMo},
we obtain the following
\begin{lem}
\label{D1} Let $1<p<\infty$. There exists a constant $c\equiv c(n,p)>0$
such that 
\begin{center}
$c^{-1}(|\xi|^{2}+|\eta|^{2})^{\frac{p-2}{2}}\leq\dfrac{\left|\vert\xi\vert^{\frac{p-2}{2}}\xi-\vert\eta\vert^{\frac{p-2}{2}}\eta\right|^{2}}{|\xi-\eta|^{2}}\leq c\,(|\xi|^{2}+|\eta|^{2})^{\frac{p-2}{2}}$ 
\par\end{center}
\noindent for every $\xi,\eta\in\mathbb{R}^{n}$ with $\xi\neq\eta$. 
\end{lem}

\selectlanguage{british}%
\noindent $\hspace*{1em}$\foreignlanguage{english}{For the function
$H_{p-1}(\xi)$ defined at (\ref{eq:Hfun}) with $\delta=p-1$, we record
the following estimates, }whose proof can be found in\foreignlanguage{english}{ \cite[Lemma 2.5]{AmG}
(with $\lambda=1$) and \cite[Lemma 2.5]{AM}}. 
\selectlanguage{english}%

\begin{lem}
\label{lem:Brasco} Let \foreignlanguage{british}{$p\in[2,\infty)$}. Then, there exists a constant $\nu\equiv \nu(p)>0$
such that 
\begin{equation}
\langle H_{p-1}(\xi)-H_{p-1}(\eta),\xi-\eta\rangle\,\geq\,\nu\,\vert H_{\frac{p}{2}}(\xi)-H_{\frac{p}{2}}(\eta)\vert^{2},\label{eq:BraAmb}
\end{equation}
for every $\xi,\eta\in\mathbb{R}^{n}$.
\end{lem}
\begin{lem}
\label{eq:BraAmb}
For every $(\theta,\varepsilon)\in \mathbb{R}^+\times\mathbb{R}^+$ with $\theta<\varepsilon,$ \foreignlanguage{british} there exist two positive constants $\beta_1(\theta,\varepsilon)$ and $\beta_2(\theta,\varepsilon, n)$
such that 
\begin{equation}
\beta_1 |H_{\varepsilon}(\xi)-H_{\varepsilon}(\eta)|\leq\,\,\frac{\vert H_{\theta}(\xi)-H_{\theta}(\eta)\vert}{\big((|\xi|-1)^\varepsilon_++(|\eta|-1)^\varepsilon_+\big)^{\frac{\theta-\varepsilon}{\varepsilon}}}\leq\beta_2|H_{\varepsilon}(\xi)-H_{\varepsilon}(\eta)|
\end{equation}
for every $\xi,\eta\in\mathbb{R}^{n}$. 
\end{lem}

\begin{lem}\label{Duzaar}
  For any $\theta > 0$, there exists a constant $c = c(\theta)$ such that, for all $\eta, \zeta \in \mathbb{R}^n \setminus \{0\}$, $n \in \mathbb{N}$, we have
\[
\frac{1}{c} \left||\eta|^{\theta - 1} \eta - |\zeta|^{\theta - 1} \zeta  \right| \leq \left( |\eta| + |\zeta| \right)^{\theta - 1} |\eta - \zeta| \leq c \left| \, |\eta|^{\theta - 1} \eta - |\zeta|^{\theta - 1} \zeta \,  \right|.
\]
\end{lem}
We conclude by recalling a well-known iteration Lemma, whose proof can be found in \cite[Lemma 6.1]{Giu}.
\begin{lem}
    \label{lem:Giusti2} Let \foreignlanguage{british}{$Z(t)$} be a bounded non-negative function in the interval $[\rho,R].$ Assume that for $\rho\leq r<t\leq R$ we have $$Z(r)\leq[A(t-r)^{-\theta}+B(t-r)^{-\sigma}+C]+\lambda Z(t)$$ with $A, B, C\geq0,\textbf{ } \theta>\sigma>0$ and $0\leq \lambda<1.$ Then,
    $$Z(\rho)\leq c(\theta,\lambda)[A(R-\rho)^{-\theta}+B(R-\rho)^{-\sigma}+C].$$
\end{lem}

\subsection{Difference quotients}

\label{subsec:DiffOpe}

\noindent $\hspace*{1em}$We recall here the definition and some elementary
properties of the difference quotients that will be useful in the
following (see, for example, \cite{Giu}). 
\begin{defn}
\noindent For every vector-valued function $F:\mathbb{R}^{n}\rightarrow\mathbb{R}^{k}$
the \textit{finite difference operator }in the direction $x_{j}$
is defined by 
\[
\tau_{j,h}F(x)=F(x+he_{j})-F(x),
\]
where $h\in\mathbb{R}$, $e_{j}$ is the unit vector in the direction
$x_{j}$ and $j\in\{1,\ldots,n\}$.\\
 $\hspace*{1em}$The \textit{difference quotient} of $F$ with respect
to $x_{j}$ is defined for $h\in\mathbb{R}\setminus\{0\}$ by 
\[
\Delta_{j,h}F(x)\,=\,\frac{\tau_{j,h}F(x)}{h}\,.
\]
\end{defn}

\noindent When no confusion can arise, we shall omit the index $j$
and simply write $\tau_{h}$ or $\Delta_{h}$ instead of $\tau_{j,h}$
or $\Delta_{j,h}$, respectively. 
\selectlanguage{british}%
\begin{prop}
\label{prop}
\noindent Let $\Omega\subset\mathbb{R}^{n}$ be an open set and let
$F\in W^{1,q}(\Omega)$, with $q\geq1$. Moreover, let $G:\Omega\rightarrow\mathbb{R}$
be a measurable function and consider the set
\[
\Omega_{\vert h\vert}:=\left\{ x\in\Omega:\mathrm{dist}\left(x,\partial\Omega\right)>\vert h\vert\right\} .
\]
\foreignlanguage{english}{Then:}\\
\foreignlanguage{english}{}\\
\foreignlanguage{english}{$\mathrm{(}\mathrm{i}\mathrm{)}$ $\Delta_{h}F\in W^{1,q}\left(\Omega_{\vert h\vert}\right)$
and $\partial_{i}(\Delta_{h}F)=\Delta_{h}(\partial_{i}F)$ for every
$\,i\in\{1,\ldots,n\}$.}\\

\selectlanguage{english}%
\noindent $\mathrm{(}\mathrm{ii}\mathrm{)}$ If at least one of the
functions $F$ or $G$ has support contained in $\Omega_{\vert h\vert}$,
then 
\[
\int_{\Omega}F\,\Delta_{h}G\,\, dx\,=\,-\int_{\Omega}G\,\Delta_{-h}F\,\, dx.
\]
$\mathrm{(}\mathrm{iii}\mathrm{)}$ We have 
\[
\Delta_{h}(FG)(x)=F(x+he_{j})\Delta_{h}G(x)\,+\,G(x)\Delta_{h}F(x).
\]
\end{prop}

\selectlanguage{english}%
\noindent The next result about the finite difference operator is
a kind of integral version of the Lagrange Theorem and its proof can
be found in \cite[Lemma 8.1]{Giu}. 
\begin{lem}
\noindent \label{lem:Giusti1} If $0<\rho<R$, $\vert h\vert<\frac{R-\rho}{2}$,
$1<q<+\infty$ and $F\in L^{q}(B_{R},\mathbb{R}^{k})$ is such that
$DF\in L^{q}(B_{R},\mathbb{R}^{k\times n})$, then 
\[
\int_{B_{\rho}}\left|\tau_{h}F(x)\right|^{q}\, \, dx\,\leq\,c^{q}(n)\,\vert h\vert^{q}\int_{B_{R}}\left|DF(x)\right|^{q}\, \, dx.
\]
Moreover 
\[
\int_{B_{\rho}}\left|F(x+he_{j})\right|^{q}\, \, dx\,\leq\,\int_{B_{R}}\left|F(x)\right|^{q}\, \, dx.
\]
\end{lem}

\noindent Finally, we recall the following fundamental result, whose
proof can be found in \cite[Lemma 8.2]{Giu}. 
\begin{lem}
\noindent \label{lem:RappIncre} Let $F:\mathbb{R}^{n}\rightarrow\mathbb{R}^{k}$,
$F\in L^{q}(B_{R},\mathbb{R}^{k})$ with $1<q<+\infty$. Suppose that
there exist $\rho\in(0,R)$ and a constant $M>0$ such that 
\[
\sum_{j=1}^{n}\int_{B_{\rho}}\left|\tau_{j,h}F(x)\right|^{q}\, \, dx\,\leq\,M^{q}\,\vert h\vert^{q}
\]
for every $h\in\mathbb{R}$ with $\vert h\vert<\frac{R-\rho}{2}$.
Then $F\in W^{1,q}(B_{\rho},\mathbb{R}^{k})$ with the estimate 
\[
\Vert DF\Vert_{L^{q}(B_{\rho})}\leq M
\]

\[
\Delta_{j,h}F\rightarrow\partial_{j}F\,\,\,\,\,\,\,\,\,\,in\,\,L_{loc}^{q}(B_{R},\mathbb{R}^{k})\,\,\,\,\mathit{as}\,\,h\rightarrow0,
\]
for each $j\in\{1,\ldots,n\}$. 
\end{lem}
For further needs, we record the following
\begin{lem}
\label{lem:sob}
    Let $\Omega\subset\mathbb{R}^n$ be a bounded open set, $p\geq 2$ and $u\in W^{1,p}_{loc}(\Omega).$ Then, for $s>1, $the implication $$H_{p/2}(Du)\in W^{1,s}_{loc}(\Omega)\Longrightarrow |Du|\in L^{\frac{nps}{2(n-s)}}_{loc}(\Omega)$$ holds true, together with the estimate
$$\bigg(\int_{B_R}|Du|^{\frac{nps}{2(n-s)}}\, dx\bigg)^{\frac{n-s}{ns}}\leq C\bigg(\int_{B_R}|DH_{p/2}(Du)|^s\,\, dx\bigg)^{\frac{1}{s}}+\frac{C}{R}\left(\int_{B_R}|H_{p/2}(Du)|^s\, dx\right)^{\frac{1}{s}}+CR^{\frac{n-s}{s}}.$$
for every $B_R\Subset\Omega$.
\end{lem}
\begin{proof}
    By the Sobolev embedding theorem we have that $$H_{p/2}(Du)\in L^{s^*}_{loc}(\Omega)$$ with the estimate $$\bigg(\int_{B_R}|H_{p/2}(Du)|^{s^*}\, dx\bigg)^{\frac{1}{s^*}}\leq C\bigg(\int_{B_R}|DH_{p/2}(Du)|^s\, dx\bigg)^{\frac{1}{s}}+\frac{C}{R}\bigg(\int_{B_R}|H_{p/2}(Du)|^s\, dx\bigg)^{\frac{1}{s}}$$  that, by the definition of $H_{p/2}(\xi)$ in \eqref{eq:Hfun}, can be written as follows
    $$\bigg(\int_{B_R}(|Du|-1)_+^{\frac{nps}{2(n-s)}}\, dx\bigg)^{\frac{n-s}{ns}}\leq C\bigg(\int_{B_R}|DH_{p/2}(Du)|^s\, dx\bigg)^{\frac{1}{s}}+\frac{C}{R}\bigg(\int_{B_R}(|Du|-1)^{\frac{ps}{2}}_+\, dx\bigg)^{\frac{1}{s}}.$$  Since \begin{eqnarray*}
     &&\bigg(\int_{B_R}|Du|^{\frac{nps}{2(n-s)}}\, dx\bigg)^{\frac{n-s}{ns}}\cr\cr 
     &=& \bigg(\int_{B_R\cap\{|Du|\le 1\}}|Du|^{\frac{nps}{2(n-s)}}\, dx+\int_{B_R\cap\{|Du|> 1\}}[(|Du|-1)_++1]^{\frac{nps}{2(n-s)}}\, dx\bigg)^{\frac{n-s}{ns}}\cr\cr
     &\le& C R^{\frac{n-s}{s}}+C\bigg(\int_{B_R}(|Du|-1)_+^{\frac{nps}{2(n-s)}}\, dx\bigg)^{\frac{n-s}{ns}},
    \end{eqnarray*}
    the conclusion easily follows.
\end{proof}
\section{A priori estimate}
This section is devoted to the proof of an \textit{a priori} estimate for the solutions to equation \eqref{equazione mia}, which is the first of the main points in our proof. More precisely, we assume a priori that $H_{p/2}(Du)\in W^{1,s}_{loc}(\Omega)$, with $1<s<2$ and we estimate the $L^s$ norm of the derivative of $H_{p/2}(Du)$ with a constant that is independent of the a priori assumption.
In particular let  
\begin{equation}
\label{cond s}
 s_0=\max\bigg\{\frac{2n}{n+\beta},\frac{2np}{np+2p-2},\frac{2nq}{p(n+q-\alpha-2)+2}\bigg\}   
\end{equation}  
and note that, by assumptions \eqref{alebe} and \eqref{q}, we have $s_0<2$ and so we are legitimate to choose $s_0<s<2.$ 
\begin{thm}
\label{teo2}
   Let $u\in W^{1,p}_{loc}(\Omega, |x|^\beta dx)$ be a weak solution of equation \eqref{equazione mia}, under assumptions \eqref{alebe} and \eqref{q}. 
    If $H_{p/2}(Du)\in W^{1,s}_{loc}(\Omega)$ with $s_0<s<2$, then there exists $R_0=R_0(n,p,q,\alpha,\beta,s)$ such that the following estimate 
\begin{align*}
\int_{B_{R/2}}|D(H_{p/2}(Du))|^s\, dx&\leq C\left(\int_{{B_{R}}}|Du|^p|x|^{\beta}\, dx\right)^{\frac{s}{p}}+C\bigg(\int_{B_{R}}|Du|^{\frac{ps}{2}}\, dx+\int_{B_{R}}|u|^{\frac{ps}{2}}\, dx+1\bigg)^{\kappa}\\
&\leq C\left(1+\int_{{B_{R}}}(|Du|^p+|u|^p)|x|^{\beta}\, dx\right)^{\kappa},
\end{align*}
    holds for every $B_R\subset B_{R_0}\Subset\Omega$, for constants $C=C(n, p, q, \alpha, \beta, R)$ and $\kappa=\kappa(p,q)>0.$
\end{thm}
\noindent \begin{proof}[\bfseries{Proof}]

Fix a ball $B_R\Subset\Omega$, that contains the origin, radii $R/2\leq \rho_1\leq r_1<r_2<\rho_2\leq R$ and let $\zeta \in C^\infty_0(B_{r_2})$ be a cut-off function such that $0\leq\zeta\leq 1, \zeta=1$ in $B_{r_1}$ and $|D\zeta|<\frac{C}{r_2-r_1}$. In what follows, without loss of generality, we shall suppose that $R<1.$ Since $u$ is a weak solution of \eqref{equazione mia}, choosing $\varphi=\tau_{j,-h}(\zeta^2\tau_{j,h}u)$ as a test function in \eqref{def}, with $h\in \mathbb{R}\setminus\{0\}, |h|<\frac{\rho_2-r_2}{2}$, integrating by parts by means of $(ii)$ in Proposition \ref{prop}, we get
$$-\int_\Omega\biggl\langle\tau_{j,h}\bigg(|x|^\beta H_{p-1}(Du)\bigg),\zeta^2D(\tau_{j,h}u)+2\zeta D\zeta\tau_{j,h}u\biggr\rangle\textbf{ }\, dx=\int_\Omega \zeta^2 \tau_{j,h}\bigg(\frac{|u|^{q-2}u}{|x|^\alpha}\bigg)\tau_{j,h}u\textbf{ } \, dx.$$ 
Using $(iii)$ in Proposition \ref{prop}, the previous equality can be written as follows
\begin{align*}
&0=\int_\Omega \zeta^2|x+he_j|^\beta\bigl\langle\tau_{j,h}(H_{p-1}(Du)),\tau_{j,h}(Du)\bigr\rangle\textbf{ }\, dx\\
&+\int_\Omega\zeta^2\bigl\langle\tau_{j,h}\big(|x|^\beta\big) H_{p-1}(Du),\tau_{j,h}(Du)\bigr\rangle\textbf{ }\, dx\\
&+2\int_\Omega\zeta\bigl\langle\tau_{j,h}\big(|x|^\beta\big) H_{p-1}(Du),\tau_{j,h}u D\zeta\bigr\rangle\textbf{ }\, dx\\
&+2\int_\Omega \zeta |x+he_j|^\beta\bigl\langle\tau_{j,h}(H_{p-1}(Du)),\tau_{j,h}u D\zeta\bigr\rangle\textbf{ }\, dx\\
&+\int_\Omega \zeta^2 \tau_{j,h}u|u(x+he_j)|^{q-2}u(x+he_j)\tau_{j,h}\bigg(\frac{1}{|x|^\alpha}\bigg)\textbf{ }\, dx\\
&+\int_\Omega \zeta^2 \frac{1}{|x|^\alpha}\tau_{j,h}(|u|^{q-2}u)\tau_{j,h}u\textbf{ }\, dx\\
& =:I_1+I_2+I_3+I_4+I_5+I_6,
\end{align*}
which yields 
\begin{equation}
\label{dis:genl}
    I_1\leq |I_2|+|I_3|+|I_4|+|I_5|+|I_6|.
\end{equation}  
Lemma \ref{lem:Brasco} yields
\begin{align}
\label{i2l}
 I_1&=\int_\Omega\zeta^2 |x+he_j|^\beta\bigl\langle\tau_{j,h}(H_{p-1}(Du)),\tau_{j,h}(Du)\bigr\rangle \, dx\geq \nu\int_{B_{r_2}}\zeta^2|x+he_j|^\beta|\tau_{j,h}(H_{p/2}(Du))|^2\, dx.
\end{align}
As observed in \cite{MP}, by $(iii)$ in Proposition \ref{prop},  in the set $\{y\in B_R: |Du(y)|>1\}$ we may write 
\begin{align*}
\nonumber
    H_{p-1}(Du)\tau_{j,h}Du&=\tau_{j,h}\Big(H_{p-1}(Du)\cdot Du\Big)-Du(x+he_j)\tau_{j,h}\Big(H_{p-1}(Du)\Big)\\ \nonumber
    &=\tau_{j,h}\Big((|Du|-1)^{p-1}_+|Du|\Big)-Du(x+he_j)\tau_{j,h}\Big(H_{p-1}(Du)\Big)\\ \nonumber
    &=\tau_{j,h}\Big((|Du|-1)^{p}_+\Big)+\tau_{j,h}\Big((|Du|-1)^{p-1}_+\Big)-Du(x+he_j)\tau_{j,h}\Big(H_{p-1}(Du)\Big)\\
    &=\tau_{j,h}\Big(|H_p(Du)|\Big)+\tau_{j,h}\Big(|H_{p-1}(Du)|\Big)-Du(x+he_j)\tau_{j,h}\Big(H_{p-1}(Du)\Big),
\end{align*} where we used the definition of $H_\delta(\xi)$ both with $\delta=p$ and $\delta=p-1$ and the linearity of the finite difference operator.\\
From the previous equality, we infer
\begin{eqnarray}\label{prodl}
 | H_{p-1}(Du)||\tau_{j,h}Du| &\le & \big|\tau_{j,h}\big(H_p(Du)\big)\big|+ \big|\tau_{j,h}\big(H_{p-1}(Du)\big)\big|\Big(1+|Du(x+he_j)|\Big)\cr\cr
 &\le & \big|\tau_{j,h}\big(H_p(Du)\big)\big|+ 2\big|\tau_{j,h}\big(H_{p-1}(Du)\big)\big||Du(x+he_j)|,
\end{eqnarray}
in the set $\{x\in B_R: |Du(x)|>1\}$.
By \eqref{prodl}, we obtain
\begin{align}
\label{i1i2i3}
\nonumber
 |I_2|&\leq\int_{B_{r_2}}\zeta^2\big|\tau_{j,h}(|x|^\beta)\big|\big|\tau_{j,h}(|H_{p}(Du)|)\big|\, dx\\\nonumber
&+2\int_{B_{r_2}}\zeta^2\big|\tau_{j,h}(|x|^\beta|)\big|\big(|Du(x+he_j)|\big|\tau_{j,h}(H_{p-1}(Du)) \big|\big)\, dx\\ \nonumber
&\leq C\int_{{B_{r_2}}}\zeta^2\big|\tau_{j,h}(|x|^\beta)\big|\Big|\tau_{j,h}H_{p/2}(Du)\Big|\Big[|H_{p/2}(Du(x+he_j))|+|H_{p/2}(Du)|\Big]\, dx, \\ \nonumber
&+C\int_{B_{r_2}}\!\!\!\zeta^2\big|\tau_{j,h}\big(|x|^\beta\big)\big||Du(x+he_j)|\big|\tau_{j,h}(H_{p/2}(Du))\big|\Big(\!|H_{p/2}(Du(x+he_j))|^{\frac{p-2}{p}}+|H_{p/2}(Du)|^{\frac{p-2}{p}}\!\Big) dx\\
&=:|I_{2,1}|+|I_{2,2}|,
\end{align} where we used Lemma \ref{eq:BraAmb} for the first integral with $\theta=p/2$, $\varepsilon=p$, $\xi=Du(x+he_j)$ and $\eta=Du(x)$ and for the second with $\theta=p/2$ and $\varepsilon=p-1$.
We now proceed estimating the integrals $|I_{2,j}|,\textbf{ }j=1,2$.\\
Since $s>s_0$, by \eqref{cond s} we have in particular that $s>\frac{2n}{n+1}$ and so we are legitimate to use H\"older's inequality with exponents $$s,\frac{ns}{n-s}, \frac{ns}{s(n+1)-2n}.$$
\begin{align*}
|I_{2,1}|&\leq C \bigg(\int_{B_{r_2}}\zeta^2\big|\tau_{j,h}H_{p/2}(Du)\big|^s\, dx\bigg)^{\frac{1}{s}}\bigg(\int_{{B_{r_2}}}\big|\tau_{j,h}(|x|^\beta)\big|^{\frac{ns}{s(n+1)-2n}}\, dx\bigg)^{\frac{s(n+1)-2n}{ns}}\\
&\qquad\quad\cdot\bigg(\int_{B_{r_2}}|H_{p/2}(Du)|^{\frac{ns}{n-s}}\, dx\bigg)^{\frac{n-s}{ns}}\\
&\leq C|h|\bigg(\int_{B_{r_2}}\zeta^2\big|\tau_{j,h}H_{p/2}(Du)\big|^s \, dx\bigg)^{\frac{1}{s}}\bigg(\int_{B_{R}}|D(|x|^\beta)|^{\frac{ns}{s(n+1)-2n}}\bigg)^{\frac{s(n+1)-2n}{ns}}\\
&\qquad\qquad\cdot\bigg(\int_{B_{r_2}}|H_{p/2}(Du)|^{\frac{ns}{n-s}}\, dx\bigg)^{\frac{n-s}{ns}},
\end{align*}
where we also used the properties of $\zeta$ and in the last inequality we used Lemma  \ref{lem:Giusti1}. Using again H\"older's inequality  and Lemma \ref{lem:sob}, we get
\begin{align*}
|I_{2,1}|&\leq C|h|\bigg(\int_{B_{r_2}}\zeta^2|x+he_j|^\beta\big|\tau_{j,h}H_{p/2}(Du)\big|^2\, dx\bigg)^{\frac{1}{2}}\bigg(\int_{B_R}|x|^{\frac{-\beta s}{2-s}}\bigg)^{\frac{2-s}{2}}\\ &\bigg(\int_{B_{R}}|D\big(|x|^\beta\big)|^{\frac{ns}{s(n+1)-2n}}\bigg)^{\frac{s(n+1)-2n}{ns}}\bigg(\int_{B_{\rho_2}}|DH_{p/2}(Du)|^s\,\, dx+\frac{1}{\rho_2^s}\int_{B_{\rho_2}}|H_{p/2}(Du)|^{s}\, dx+\rho_2^{n-s}\bigg)^{\frac{1}{s}}.
\end{align*}
Since, for $\vartheta>-\frac{n}{\ell}$ it holds
\begin{equation}
    \label{x}
    \left(\int_{B_R}|x|^{\vartheta\ell}\, \, dx\right)^{\frac{1}{\ell}}\leq C(\ell, n) R^{\vartheta+\frac{n}{\ell}},
    \end{equation}
considering that $|D(|x|^\beta)|\approx c(\beta)|x|^{\beta-1}$ and since, by \eqref{cond s} $s>s_0>\frac{2n}{n+\beta}$,
we get 
\begin{align}
\label{i11} \nonumber
|I_{2,1}|&\leq C|h| R^{\frac{2-s}{2}\left(\beta-\frac{n(2-s)}{s}\right)}\bigg(\int_{B_{r_2}}\zeta^2|x+he_j|^\beta\big|\tau_{j,h}H_{p/2}(Du)\big|^2\, dx\bigg)^{\frac{1}{2}}\\ \nonumber
&\qquad\qquad \bigg(\int_{B_{\rho_2}}|DH_{p/2}(Du)|^s\,\, dx+\frac{1}{\rho_2^s}\int_{B_{\rho_2}}|H_{p/2}(Du)|^{s}\, dx+\rho_2^{n-s}\bigg)^{\frac{1}{s}}\\ \nonumber
&\leq \varepsilon\int_{B_{r_2}}\zeta^2|x+he_j|^\beta\big|\tau_{j,h}H_{p/2}(Du)\big|^2\, dx\\ \nonumber
&+ C_\varepsilon|h|^2 R^{\left(\beta(2-s)-\frac{n(2-s)^2}{s}\right)} \bigg(\int_{B_{\rho_2}}|DH_{p/2}(Du)|^s\,\, dx\bigg)^{\frac{2}{s}}\\
&+C_\varepsilon \frac{|h|^2}{\rho_2^{2}}\bigg(\int_{B_{\rho_2}}(1+|H_{p/2}(Du)|^{s})\, dx\bigg)^{\frac{2}{s}},
\end{align}
where we also used Young's inequality, the fact that the exponent $$\left(\beta(2-s)-\frac{n(2-s)^2}{s}\right)>0$$ and that $R<1$.

Using the properties of $\zeta$, that $|H_{p/2}(\xi)|\le |\xi|^{p/2}$ and H\"older's inequality with the same exponents as for the estimate of $I_{2,1}$, we get
\begin{align*}
 \nonumber
|I_{2,2}|&\leq C\int_{B_{r_2}}\zeta^2\big|\tau_{j,h}\big(|x|^\beta\big)\big|\big|\tau_{j,h}(H_{p/2}(Du))\big|\Big(|Du(x+he_j)|+|Du(x)|\Big)^{\frac{p}{2}}\, dx\\
&\leq 2\left(\int_{B_{r_2}} \zeta^2|\tau_{j,h}H_{p/2}(Du)|^s \, dx\right)^{\frac{1}{s}} \left(\int_{B_{r_2}}|\tau_{j,h}\big(|x|^\beta\big)|^{\frac{ns}{s(n+1)-2n}}\right)^{\frac{s(n+1)-2n}{ns}}\cr\cr
&\qquad\cdot\left(\int_{B_{r_2}}\Big(|Du(x+he_j)|+|Du(x)|\Big)^{\frac{p}{2}\frac{ns}{n-s}}\, dx\right)^{\frac{n-s}{ns}}.
\end{align*}
By Lemma \ref{lem:Giusti1}, we obtain
\begin{eqnarray*}\label{i12} 
|I_{2,2}|&\leq& C|h|\left(\int_{B_{r_2}} \zeta^2|\tau_{j,h}H_{p/2}(Du)|^s\, dx\right)^{\frac{1}{s}} \left(\int_{B_{R}}|D\big(|x|^\beta\big)|^{\frac{ns}{ns-2n+s}}\right)^{\frac{ns-2n+s}{ns}}\left(\int_{B_{\rho_2}}|Du|^{\frac{p}{2}\frac{ns}{n-s}}\, dx\right)^{\frac{n-s}{ns}}\cr\cr 
&\leq& C|h|\bigg(\int_{B_{r_2}}\zeta^2|x+he_j|^\beta\big|\tau_{j,h}H_{p/2}(Du)\big|^2\, dx\bigg)^{\frac{1}{2}}\bigg(\int_{B_R}|x|^{\frac{-\beta s}{2-s}}\bigg)^{\frac{2-s}{2}}\cr\cr
&&\qquad\qquad\cdot\left(\int_{B_{r_2}}|D\big(|x|^\beta\big)|^{\frac{ns}{s(n+1)-2n}}\right)^{\frac{s(n+1)-2n}{ns}}\left(\int_{B_{\rho_2}}|Du|^{\frac{p}{2}\frac{ns}{n-s}}\, dx\right)^{\frac{n-s}{ns}}\cr\cr 
&\leq& C|h| R^{\frac{2-s}{2}\left(\beta-\frac{n(2-s)}{s}\right)}\bigg(\int_{B_{r_2}}\zeta^2|x+he_j|^\beta\big|\tau_{j,h}H_{p/2}(Du)\big|^2\, dx\bigg)^{\frac{1}{2}}\left(\int_{B_{\rho_2}}|Du|^{\frac{p}{2}\frac{ns}{n-s}}\, dx\right)^{\frac{n-s}{ns}},
\end{eqnarray*}
where as before we used H\"older's inequality and \eqref{x}.  Hence by Lemma \ref{lem:sob} and Young's inequality, we deduce that
\begin{eqnarray}\label{i12} 
|I_{2,2}|&\leq&  C|h| R^{\frac{2-s}{2}\left(\beta-\frac{n(2-s)}{s}\right)}\bigg(\int_{B_{r_2}}\zeta^2|x+he_j|^\beta\big|\tau_{j,h}H_{p/2}(Du)\big|^2\, dx\bigg)^{\frac{1}{2}}\cr\cr
&&\qquad\qquad \left(\int_{B_{\rho_2}}|DH_{p/2}(Du)|^s\,\, dx+\frac{1}{\rho_2^s}\int_{B_{\rho_2}}|H_{p/2}(Du)|^{s}\, dx+\rho_2^{n-s}\right)^{\frac{1}{s}}\cr\cr
&\leq& \varepsilon\int_{B_{r_2}}\zeta^2|x+he_j|^\beta\big|\tau_{j,h}H_{p/2}(Du)\big|^2\, dx \cr\cr
&+& C_\varepsilon|h|^2 R^{\left(\beta(2-s)-\frac{n(2-s)^2}{s}\right)}\bigg(\int_{B_{r_2}}|DH_{p/2}(Du)|^s\,\, dx\bigg)^{\frac{2}{s}}\cr\cr
&+&C_\varepsilon \frac{|h|^2}{\rho_2^{2}}\bigg(\int_{B_{\rho_2}}(1+|H_{p/2}(Du)|^{s})\, dx\bigg)^{\frac{2}{s}},
\end{eqnarray} where we argued as in \eqref{i11} and $\varepsilon>0$ will be chosen later.
\\
Inserting \eqref{i11} and  \eqref{i12} in \eqref{i1i2i3}, we obtain
\begin{align}
\label{i1} \nonumber
|I_2|&\leq 2\varepsilon\int_{B_{r_2}}\zeta^2|x+he_j|^\beta\big|\tau_{j,h}H_{p/2}(Du)\big|^2\, dx \\ \nonumber
&+C_\varepsilon|h|^2 R^{\left(\beta(2-s)-\frac{n(2-s)^2}{s}\right)}\bigg(\int_{B_{\rho_2}}|DH_{p/2}(Du)|^s\,\, dx\bigg)^{\frac{2}{s}}\\
&+C_\varepsilon \frac{|h|^2}{\rho_2^{2}}\bigg(\int_{B_{\rho_2}}(1+|H_{p/2}(Du)|^{s})\, dx\bigg)^{\frac{2}{s}}.
\end{align}
We now proceed with the estimate of $|I_3|$. By the properties of $|D\zeta|$ we have
\begin{align*}
|I_3|
&\leq \frac{C}{r_2-r_1}\int_{B_t}\big|\tau_{j,h}\big(|x|^\beta\big)\big||\tau_{j,h}u|(|Du|-1)^{p-1}_+\, dx.
\end{align*} 
Since $s>s_0$, by \eqref{cond s}, we have in particular that $s>\frac{2np}{np+2p-2}$ and so we can use H\"older's inequality with exponents $$ \frac{ps}{2}, \frac{nps}{2(n-s)(p-1)}, \frac{nps}{2s(p-1)-np(2-s)}$$ to deduce that
\begin{align*}
|I_3|&\leq \frac{C}{r_2-r_1}\bigg(\int_{B_{r_2}}\big|\tau_{j,h}\big(|x|^\beta\big)\big|^{\frac{nps}{2s(p-1)-np(2-s)}}\, dx\bigg)^{\frac{2s(p-1)-np(2-s)}{nps}}\bigg(\int_{B_{r_2}}|\tau_{j,h}u|^{\frac{ps}{2}} \, dx\bigg)^{\frac{2}{ps}}\\\nonumber
&\qquad\qquad\cdot\bigg(\int_{B_{r_2}}(|Du|-1)_+^{\frac{nps}{2(n-s)}}\, dx\bigg)^{\frac{2(n-s)(p-1)}{nps}}\\   \nonumber
&\leq \frac{C|h|^2}{r_2-r_1}\bigg(\int_{B_R}\big|D\big(|x|^\beta\big)\big|^{\frac{nps}{2s(p-1)-np(2-s)}}\, dx\bigg)^{\frac{2s(p-1)-np(2-s)}{nps}}\bigg(\int_{B_R}|Du|^{\frac{ps}{2}}\, dx\bigg)^{\frac{2}{ps}}\\\nonumber 
&\qquad\qquad\cdot\bigg(\int_{B_{\rho_2}}|DH_{p/2}(Du)|^s\,\, dx+\frac{1}{\rho_2^s}\int_{B_{\rho_2}}|H_{p/2}(Du)|^{s}\, dx+\rho_2^{n-s}\bigg)^{\frac{2(p-1)}{ps}}\\   \nonumber
&\leq \frac{CR^{\beta-1+\frac{2s(p-1)-np(2-s)}{ps}}|h|^2}{r_2-r_1}\bigg(\int_{B_R}|Du|^{\frac{ps}{2}}
\, dx\bigg)^{\frac{2}{ps}}\\\nonumber 
&\qquad\qquad\cdot\bigg(\int_{B_{\rho_2}}|DH_{p/2}(Du)|^s\,\, dx+\frac{1}{\rho_2^s}\int_{B_{\rho_2}}|H_{p/2}(Du)|^{s}\, dx+\rho_2^{n-s}\bigg)^{\frac{2(p-1)}{ps}}
\\ \nonumber
&\leq\sigma|h|^2\bigg(\int_{B_{\rho_2}}|DH_{p/2}(Du)|^s\,\, dx+\frac{1}{\rho_2^2}\int_{B_{\rho_2}}|H_{p/2}(Du)|^{s}\, dx+\rho_2^{n-s}\bigg)^{\frac{2}{s}}\\ 
&+\frac{C_\sigma|h|^2}{(r_2-r_1)^p}R^{(\beta-1)p+\frac{2s(p-1)-np(2-s)}{s}}\left(\int_{B_R}|Du|^{{\frac{ps}{2}}}\, dx\right)^{\frac{2}{s}},
\end{align*}
where  we used in turn Lemmas \ref{lem:Giusti1}, \ref{lem:sob}, \eqref{x} and Young's inequality and $\sigma>0$ will be chosen later.\\
Using again that $s>s_0,$ by \eqref{cond s} we have $$s>\frac{2np}{np+\beta p+p-2}$$ which yields $$(\beta-1)p+\frac{2s(p-1)-np(2-s)}{s}>0$$ and so arguing as in \eqref{i11}, previous estimate simplifies to
\begin{eqnarray}
    \label{i3}
    |I_3|&\leq& \sigma|h|^2\bigg(\int_{B_{\rho_2}}|DH_{p/2}(Du)|^s\,\, dx\bigg)^{\frac{2}{s}}+\sigma|h|^2\bigg(\frac{C}{\rho_2^2}\int_{B_{\rho_2}}(1+|H_{p/2}(Du)|^{s})\, dx\bigg)^{\frac{2}{s}}\cr\cr 
&&+\frac{C_\sigma|h|^2}{(r_2-r_1)^p}\left(\int_{B_R}|Du|^{\frac{ps}{2}}\, dx\right)^{\frac{2}{s}}.
\end{eqnarray}
For the estimate of $|I_4|,$ we use that $|D\zeta|\leq\frac{C}{r_2-r_1}$, Lemma \ref{eq:BraAmb} with $\theta=p/2$ and $\varepsilon=p-1$ and Young's inequality as follows
\begin{eqnarray}
\label{i4 inter} 
|I_4|&\leq& \frac{C}{r_2-r_1}\int_{B_{r_2}}\zeta |x+he_j|^\beta|\tau_{j,h}(H_{p-1}(Du))||\tau_{j,h}u|\textbf{ }\, dx\cr\cr 
 &\leq& \frac{C}{r_2-r_1}\int_{B_{r_2}}\!\!\!\zeta|x+he_j|^\beta\tau_{j,h}(H_{p/2}(Du))|\Big(|H_{p/2}(Du(x+he_j))|^{\frac{p-2}{p}}\!+\!|H_{p/2}(Du(x))|^{\frac{p-2}{p}}\Big)|\tau_{j,h}u|\textbf{ }\, dx \cr\cr
 &\leq& \varepsilon\int_{B_{r_2}} \zeta^{2}|x+he_j|^\beta|\tau_{j,h}(H_{p/2}(Du))|^2\, dx\cr\cr
&&+\frac{C_\varepsilon }{(r_2-r_1)^2} \int_{B_{B_{r_2}}}|x+he_j|^\beta\Big(|H_{p/2}(Du(x+he_j))|^{\frac{p-2}{p}}+|H_{p/2}(Du(x))|^{\frac{p-2}{p}}\Big)^2|\tau_{j,h}u|^2\, dx\cr\cr
&\leq &\varepsilon\int_{B_{r_2}} \zeta^{2}|x+he_j|^\beta|\tau_{j,h}(H_{p/2}(Du))|^2\, dx\cr\cr
&&+\frac{C_\varepsilon }{(r_2-r_1)^2} \left(\int_{{B_{r_2}}}|x+he_j|^\beta\Big(|H_{p/2}(Du(x+he_j))|+|H_{p/2}(Du(x))|\Big)^2\,\, dx\right)^{\frac{p-2}{p}}\cr\cr &&\cdot\left(\int_{{B_{r_2}}}|x+he_j|^\beta|\tau_{j,h}u|^p\, dx\right)^{\frac{2}{p}},
\end{eqnarray}
where we used H\"older's inequality.
The last integral in the previous estimate can be bounded using Proposition \ref{prop} and Lemma \ref{lem:Giusti1}, as follows
\begin{align*}
 \nonumber
&\left(\int_{{B_{r_2}}}|x+he_j|^\beta|\tau_{j,h}u|^p\, dx\right)^{\frac{2}{p}}=\left(\int_{{B_{r_2}}}\Big|\tau_{j,h}(|x|^{\frac{\beta}{p}}u)- u(x)\tau_{j,h}(|x|^{\frac{\beta}{p}})\Big|^p\, dx\right)^{\frac{2}{p}}\\ \nonumber
&\leq C\left(\int_{{B_{r_2}}}|\tau_{j,h}(|x|^{\frac{\beta}{p}}u)|^p\, dx\right)^{\frac{2}{p}}+C\left(\int_{{B_{r_2}}} |u\tau_{j,h}(|x|^{\frac{\beta}{p}})|^p\, dx\right)^{\frac{2}{p}}.
\end{align*}
To estimate the first integral we use Lemma \ref{lem:Giusti1} and for the second we use H\"older's inequality  with exponents $$\frac{nps}{2n-ps} \text{ and }\frac{nps}{ns-2n+ps},$$ which is legitimate since $s>s_0$ and hence by \eqref{cond s} we also have that $s>\frac{2n}{n+p}$. This yields
\begin{align}
\label{xx}\nonumber
&\left(\int_{{B_{r_2}}}|x+he_j|^\beta|\tau_{j,h}u|^p\, dx\right)^{\frac{2}{p}}\leq C|h|^2\left(\int_{{B_{\rho_2}}}\Big|D(|x|^{\frac{\beta}{p}}u)\Big|^p\, dx\right)^{\frac{2}{p}}\\ \nonumber &+C\left(\int_{{B_{r_2}}} |u|^{\frac{nps}{2n-ps}}\, dx\right)^{\frac{2n-ps}{ns}\frac{2}{p}}\left(\int_{{B_{r_2}}} |\tau_{j,h}(|x|^{\frac{\beta}{p}})|^{\frac{nps}{ns-2n+ps}}\, dx\right)^{\frac{ns-2n+ps}{ns}\frac{2}{p}}\\ \nonumber
&\leq C|h|^2\left(\int_{{B_{\rho_2}}}|u|^p|D(|x|^{\frac{\beta}{p}})|^p\, dx\right)^{\frac{2}{p}}+ C|h|^2\left(\int_{{B_{\rho_2}}}|Du|^p|x|^{\beta}\, dx\right)^{\frac{2}{p}}\\ \nonumber
&+C(\rho_2)|h|^2\left(\int_{{B_{\rho_2}}} (|Du|^{\frac{ps}{2}}+|u|^{\frac{ps}{2}})\, dx\right)^{\frac{4}{ps}}\left(\int_{{B_{\rho_2}}} |D(|x|^{\frac{\beta}{p}})|^{\frac{nps}{ns-2n+ps}}\, dx\right)^{\frac{ns-2n+ps}{ns}\frac{2}{p}}\\ \nonumber
&\leq C|h|^2\left(\int_{{B_{\rho_2}}} |u|^{\frac{nps}{2n-ps}}\, dx\right)^{\frac{2n-ps}{ns}\frac{2}{p}}\left(\int_{{B_{\rho_2}}} |D(|x|^{\frac{\beta}{p}})|^{\frac{nps}{ns-2n+ps}}\, dx\right)^{\frac{ns-2n+ps}{ns}\frac{2}{p}}\\ \nonumber
&+ C|h|^2\left(\int_{{B_{\rho_2}}}|Du|^p|x|^{\beta}\, dx\right)^{\frac{2}{p}}\\ \nonumber
&+C|h|^2\left(\int_{{B_{\rho_2}}} (|Du|^{\frac{ps}{2}}+|u|^{\frac{ps}{2}})\, dx\right)^{\frac{4}{ps}}\left(\int_{{B_{\rho_2}}} |D(|x|^{\frac{\beta}{p}})|^{\frac{nps}{ns-2n+ps}}\, dx\right)^{\frac{ns-2n+ps}{ns}\frac{2}{p}}\\
&\leq C|h|^2\left(\int_{{B_{\rho_2}}}|Du|^p|x|^{\beta}\, dx\right)^{\frac{2}{p}}+C|h|^2 R^{\frac{2(\beta +n-2)}{p}}\left(\int_{{B_{\rho_2}}} (|Du|^{\frac{ps}{2}}+|u|^{\frac{ps}{2}})\, dx\right)^{\frac{4}{ps}},
\end{align} where we also used Sobolev inequality.\\
Using \eqref{xx}, \eqref{x}, the fact that $\frac{2(\beta +n-2)}{p}>0$ and that $R<1$ in \eqref{i4 inter}, we get 
\begin{align}
\label{i4l} \nonumber
|I_4|&\leq \varepsilon\int_{B_{r_2}} \zeta^{2}|x+he_j|^\beta|\tau_{j,h}(H_{p/2}(Du))|^2\, dx\\ \nonumber
&+\frac{C_\varepsilon }{(r_2-r_1)^2} \left(\int_{{B_{r_2}}}|x+he_j|^\beta\Big(|H_{p/2}(Du(x+he_j))|+|H_{p/2}(Du(x))|\Big)^2\,\, dx\right)^{\frac{p-2}{p}}\\ \nonumber 
&\cdot \bigg[C|h|^2\left(\int_{{B_{\rho_2}}}|Du|^p|x|^{\beta}\, dx\right)^{\frac{2}{p}}+C|h|^2\left(\int_{{B_{\rho_2}}} (|Du|^{\frac{ps}{2}}+|u|^{\frac{ps}{2}})\, dx\right)^{\frac{4}{ps}}\bigg]\\ \nonumber
&\leq \varepsilon\int_{B_{r_2}} \zeta^{2}|x+he_j|^\beta|\tau_{j,h}(H_{p/2}(Du))|^2\, dx\\ 
&+\frac{C_\varepsilon |h|^2}{(r_2-r_1)^2}\left(\int_{{B_{\rho_2}}}|Du|^p|x|^{\beta}\, dx\right)^{\frac{2}{p}}+\frac{C_\varepsilon |h|^2}{(r_2-r_1)^2}\left(1+\int_{{B_{\rho_2}}}(|u|^{\frac{ps}{2}}+|Du|^{\frac{ps}{2}})\, dx\right)^\kappa.
\end{align}
To estimate $|I_5|$, we note that, by our choice of $s$ we have $$2>s>\frac{2nq}{np+2+p(q-1)}$$ that is legitimate by \eqref{cond s}. This yields  $$ \frac{q-1}{{\left(\frac{ps}{2}\right)}^*}+\frac{2(n-s)}{nps}<1$$ that allows us to apply H\"older's inequality with exponents $$\frac{nps}{2(n-s)}, \frac{{\left(\frac{ps}{2}\right)}^*}{q-1},\frac{nps}{nps{-s(p-2)-q(2n-ps)}},$$ thus getting 
\begin{align*}
|I_5|&\leq \int_\Omega \zeta^2 |\tau_{j,h}u||u(x+he_j)|^{q-1}\bigg|\tau_{j,h}\bigg(\frac{1}{|x|^\alpha}\bigg)\bigg|\textbf{ }\, dx\\
&\leq \bigg(\int_{B_{r_2}}|\tau_{j,h}u|^{\frac{nps}{2(n-s)}}\, dx\bigg)^{\frac{2(n-s)}{nps}}\bigg(\int_{B_{\rho_2}}|u|^{{\left(\frac{ps}{2}\right)}^*}\, dx\bigg)^{\frac{q-1}{{\left(\frac{ps}{2}\right)}^*}}\\
&\qquad\cdot\bigg(\int_{B_{r_2}}\bigg|\tau_{j,h}\bigg(\frac{1}{|x|^\alpha}\bigg)\bigg|^{\frac{nps}{nps{-s(p-2)-q(2n-ps)}}} \, dx\bigg)^{\frac{nps{-s(p-2)-q(2n-ps)}}{nps}}\\ 
&\leq |h|^2 \bigg(\int_{B_{\rho_2}}|Du|^{\frac{nps}{2(n-s)}}\, dx\bigg)^{\frac{2(n-s)}{nps}}\bigg(\int_{B_{\rho_2}}|u|^{{\left(\frac{ps}{2}\right)}^*}\, dx\bigg)^{\frac{q-1}{{\left(\frac{ps}{2}\right)}^*}}\\
&\qquad\cdot\bigg(\int_{B_{\rho_2}}\frac{1}{|x|^{(\alpha+1)\frac{nps}{nps-s(p-2)-q(2n-ps)}}}\, dx\bigg)^{\frac{nps-s(p-2)-q(2n-ps)}{nps}},
\end{align*}
where we used Lemma \ref{lem:Giusti1} twice. 
Therefore, we can estimate $|I_5|$ further by the use of Lemma \ref{lem:sob}, Sobolev embedding theorem, Young's inequality and \eqref{x}, thus getting
\begin{align}
 \label{i5ll} \nonumber   
|I_5| &\leq CR^{\frac{nps{-s(p-2)-q(2n-ps)}}{ps}-(\alpha+1)}|h|^2\bigg(\int_{B_{\rho_2}}|DH_{p/2}(Du)|^s\,\, dx+\frac{1}{\rho_2^s}\int_{B_{\rho_2}}|H_{p/2}(Du)|^{s}\, dx+\rho_2^{n}\bigg)^{\frac{2}{p}\frac{1}{s}}\\\nonumber
&\qquad\cdot\bigg(\int_{B_{R}}|Du|^{{\frac{ps}{2}}}\, dx+\frac{1}{R^{{\frac{ps}{2}}}}\int_{B_{R}}|u|^{{\frac{ps}{2}}}\, dx+R^{n-{\frac{ps}{2}}}\bigg)^{\frac{q-1}{{\frac{ps}{2}}}}\\ \nonumber
&\leq\sigma|h|^2\bigg(\int_{B_{\rho_2}}|DH_{p/2}(Du)|^s\,\, dx\bigg)^{\frac{2}{s}}+\sigma |h|^2\bigg(\frac{1}{\rho_2^s}\int_{B_{\rho_2}}(1+|H_{p/2}(Du)|^{s})\, dx\bigg)^{\frac{2}{s}}\\
&+C_\sigma|h|^2\bigg(1+\int_{B_{R}}(|u|^{{\frac{ps}{2}}}+|Du|^{{\frac{ps}{2}}})\, dx\bigg)^{\frac{2(q-1)}{s(p-1)}},
\end{align} where we argued as in \eqref{i11}.
For the estimate of $|I_6|$, we have to study two different cases, $$1<q\leq 2 \text{ and }q> 2.$$ For the case $q> 2$  we use Lemma \ref{Duzaar}  with $\theta=q-1$ and the fact that, since $s>s_0$, by \eqref{cond s} we have $$2>s>\frac{2qn}{np+4+p(q-2)}.$$ This yields  $$ \frac{q-2}{(\frac{ps}{2})^*}+\frac{4(n-s)}{nps}<1$$ that allows us to apply H\"older's inequality with exponents
 $$\frac{nps}{4(n-s)},\, \frac{(\frac{ps}{2})^*}{q-2},\,\frac{nps}{nps-(q-2)(2n-ps)-2(n-s)},$$  to obtain
\begin{align}
\label{i6}
\nonumber
|I_6|&\leq \int_\Omega\zeta^2\frac{1}{|x|^\alpha}|\tau_{j,h}u|^2(|u(x+h)|+|u(x)|)^{q-2}\, dx \\ \nonumber
&\leq \bigg(\int_{B_{r_2}}|\tau_{j,h}u|^{\frac{nps}{2(n-s)}}\, dx\bigg)^{\frac{4(n-s)}{nps}}\bigg(\int_{B_{r_2}}|u(x+h)+u(x)|^{(\frac{ps}{2})^*}\, dx\bigg)^{\frac{q-2}{(\frac{ps}{2})^*}}\\ \nonumber
&\qquad\quad\cdot\bigg(\int_{B_{r_2}}\frac{1}{|x|^{\alpha\frac{nps}{nps-(q-2)(2n-ps)-2(n-s)}}}\, dx\bigg)^{\frac{nps-(q-2)(2n-ps)-2(n-s)}{nps}}\\ \nonumber
&\leq C|h|^2R^{\frac{nps-(q-2)(2n-ps)-2(n-s)}{ps}-\alpha}\bigg(\int_{B_{\rho_2}}|Du|^{\frac{nps}{2(n-s)}}\, dx\bigg)^{\frac{2(n-s)}{nps}}\\\nonumber
&\qquad\quad\cdot\bigg(\int_{B_{R}}|Du|^{{\frac{ps}{2}}}\, dx+\frac{1}{R^{{\frac{ps}{2}}}}\int_{B_{R}}|u|^{{\frac{ps}{2}}}\, dx+R^{n-{\frac{ps}{2}}}\bigg)^{\frac{q-2}{\frac{ps}{2}}}\\ \nonumber
&\leq C|h|^2R^{\frac{nps-(q-2)(2n-ps)-2(n-s)}{ps}-\alpha}\bigg(\int_{B_{\rho_2}}|DH_{p/2}(Du)|^s\,\, dx+\frac{1}{\rho_2^s}\int_{B_{\rho_2}}|H_{p/2}(Du)|^{s}\, dx+\rho_2^{n-s}\bigg)^{\frac{4}{p}\frac{1}{s}}\\\nonumber
&\qquad\quad\cdot\bigg(\int_{B_{R}}|Du|^{{\frac{ps}{2}}}\, dx+\frac{1}{R^{{\frac{ps}{2}}}}\int_{B_{R}}|u|^{{\frac{ps}{2}}}\, dx+R^{n-{\frac{ps}{2}}}\bigg)^{\frac{q-2}{\frac{ps}{2}}}\\ \nonumber
&\leq\sigma|h|^2\bigg(\int_{B_{\rho_2}}|DH_{p/2}(Du)|^s\,\, dx\bigg)^{\frac{2}{s}}+\sigma|h|^2\bigg(\frac{1}{\rho_2^s}\int_{B_{\rho_2}}(1+|H_{p/2}(Du)|^{s})\, dx\bigg)^{\frac{2}{s}}\\
&+C_\sigma|h|^2\bigg(1+\int_{B_{R}}(|u|^{{\frac{ps}{2}}}+|Du|^{{\frac{ps}{2}}})\, dx\bigg)^{\frac{2(q-2)}{s(p-2)}},
\end{align}
where we also used the properties of $\zeta$, Lemmas \ref{lem:Giusti1}, \ref{lem:sob} and Young's inequality.\\
For the case $1< q \leq 2$ we note that $$|u(x+h)-u(x)|(|u(x+h)|^{q-2}u-|u(x)|^{q-2}u(x))\leq |u(x+h)-u(x)|(|u(x+h)|^{q-1}+|u(x)|^{q-1}),$$ using the previous inequality, we get
\begin{align*}
|I_6|&\leq \int_\Omega \zeta^2 \frac{1}{|x|^\alpha}|\tau_{j,h}u|(|u(x+he_j)|+|u(x)|)^{q-1}dx \\
&\leq \bigg(\int_{B_{r_2}}|\tau_{j,h}u|^{\frac{nps}{2(n-s)}}\, dx\bigg)^{\frac{2(n-s)}{nps}}\bigg(\int_{B_{\rho_2}}|u|^{{\left(\frac{ps}{2}\right)}^*}\, dx\bigg)^{\frac{q-1}{{\left(\frac{ps}{2}\right)}^*}}\\
&\qquad\cdot\bigg(\int_{B_{r_2}}\frac{1}{|x|^\alpha}^{\frac{nps}{nps{-s(p-2)-q(2n-ps)}}} \, dx\bigg)^{\frac{nps{-s(p-2)-q(2n-ps)}}{nps}}\\ 
&\leq |h|^2R^{\frac{nps{-s(p-2)-q(2n-ps)}}{ps}-\alpha} \bigg(\int_{B_{\rho_2}}|Du|^{\frac{nps}{2(n-s)}}\, dx\bigg)^{\frac{2(n-s)}{nps}}\bigg(\int_{B_{\rho_2}}|u|^{{\left(\frac{ps}{2}\right)}^*}\, dx\bigg)^{\frac{q-1}{{\left(\frac{ps}{2}\right)}^*}},
\end{align*}where we used H\"older's inequality with the same exponents used in $I_5$ and Lemma \ref{lem:Giusti1} twice. We can estimate $I_6$ further arguing as in \eqref{i5ll}, we get
\begin{align*}
|I_6|&\leq\sigma|h|^2\bigg(\int_{B_{\rho_2}}|DH_{p/2}(Du)|^s\,\, dx\bigg)^{\frac{2}{s}}+\sigma |h|^2\bigg(\frac{1}{\rho_2^s}\int_{B_{\rho_2}}(1+|H_{p/2}(Du)|^{s})\, dx\bigg)^{\frac{2}{s}}\\
&+C_\sigma|h|^2\bigg(1+\int_{B_{R}}(|u|^{{\frac{ps}{2}}}+|Du|^{{\frac{ps}{2}}})\, dx\bigg)^{\frac{2(q-1)}{s(p-1)}}.   
\end{align*}
Next, we insert the estimates  \eqref{i2l}, \eqref{i1}, \eqref{i3}, \eqref{i4l}, \eqref{i5ll} and \eqref{i6} in \eqref{dis:genl}, to get
\begin{align*}
\nonumber
&\nu\int_\Omega\zeta^2|x+he_j|^\beta|\tau_{j,h}(H_{p/2}(Du))|^2\, dx\leq 3\varepsilon\int_\Omega\zeta^2|x+he_j|^\beta|\tau_{j,h}H_{p/2}(Du)|^2\, dx \\\nonumber 
&+|h|^2\bigg(C_\varepsilon R^{\left(\beta(2-s)-\frac{n(2-s)^2}{s}\right)}+3\sigma\bigg)\bigg(\int_{B_{r_2}}|D(H_{p/2}(Du))|^s\, dx\bigg)^{\frac{2}{s}}\\ 
&+\frac{C_\sigma|h|^2}{(r_2-r_1)^p}\bigg(1+\int_{B_{R}}(|u|^{{\frac{ps}{2}}}+|Du|^{{\frac{ps}{2}}})\, dx\bigg)^{\kappa}+\frac{C_\sigma|h|^2}{(r_2-r_1)^p}\bigg(\int_{B_R}|Du|^p|x|^\beta \, dx\bigg)^\frac{2}{p},
\end{align*} where $\kappa=\kappa(p,q)>0$.\\
Choosing $\varepsilon=\frac{\nu}{6}$, we can reabsorb the first integral in the right-hand side of the previous estimate by the left-hand side, using that $\zeta=1$ in $B_{r_1}$ and noting that $\rho_1<r_1$, we get
\begin{align}
\label{est:quasifinale}\nonumber
&\int_{B_{\rho_1}}|x+he_j|^\beta|\tau_{j,h}(H_{p/2}(Du))|^2\, dx\\\nonumber
&\leq |h|^2\bigg(\frac{CR^{\left(\beta(2-s)-\frac{n(2-s)^2}{s}\right)}}{\nu}+\frac{6\sigma}{\nu}\bigg)\bigg(\int_{B_{\rho_2}}|D(H_{p/2}(Du))|^s\, dx\bigg)^{\frac{2}{s}}\\
&+\frac{C_\sigma|h|^2}{(r_2-r_1)^p}\bigg(1+\int_{B_{R}}(|u|^{{\frac{ps}{2}}}+|Du|^{{\frac{ps}{2}}})\, dx\bigg)^{\kappa}+\frac{C_\sigma|h|^2}{(r_2-r_1)^p}\bigg(\int_{B_R}|Du|^p|x|^\beta \, dx\bigg)^\frac{2}{p}.
\end{align} 
Now, H\"older's inequality yields
\begin{eqnarray}\label{dis:s}
   \int_{B_{\rho_1}}|\tau_{j,h}(H_{p/2}(Du))|^s\, dx&\le&\left(\int_{B_{\rho_1}}|x+he_j|^\beta|\tau_{j,h}(H_{p/2}(Du))|^2\, dx\right)^{\frac{s}{2}} \bigg(\int_{B_R}|x+he_j|^{\frac{-\beta s}{2-s}}\bigg)^{\frac{2-s}{2}} \cr\cr
   &\le& C R^{-\frac{\beta s}{2}+\frac{n(2-s)}{2}}\left(\int_{B_{\rho_1}}|x+he_j|^\beta|\tau_{j,h}(H_{p/2}(Du))|^2\, dx\right)^{\frac{s}{2}}\cr\cr
&\le&C \left(\int_{B_{\rho_1}}|x+he_j|^\beta|\tau_{j,h}(H_{p/2}(Du))|^2\, dx\right)^{\frac{s}{2}},
\end{eqnarray}
where we used that $R<1$ and that since $s>s_0$ by \eqref{cond s}, we have that $-\frac{\beta s}{2}+\frac{n(2-s)}{2}>0$.
Using \eqref{dis:s} to estimate the right-hand side of \eqref{est:quasifinale}, we get
\begin{eqnarray*}
&&\int_{B_{\rho_1}}|\tau_{j,h}(H_{p/2}(Du))|^s\, dx\leq |h|^s\bigg( \frac{CR^{\left(\beta(2-s)-\frac{n(2-s)^2}{s}\right)}}{\nu}+\frac{6\sigma}{\nu}\bigg)^{\frac{s}{2}}\int_{B_{\rho_2}}|D(H_{p/2}(Du))|^s\, dx\cr\cr
&&+\frac{C_\sigma|h|^s}{(r_2-r_1)^{\frac{ps}{2}}}\bigg(1+\int_{B_{R}}(|u|^{{\frac{ps}{2}}}+|Du|^{{\frac{ps}{2}}})\, dx\bigg)^{\kappa}+\frac{C_\sigma|h|^s}{(r_2-r_1)^{\frac{ps}{2}}}\bigg(\int_{B_R}|Du|^p|x|^\beta \, dx\bigg)^\frac{s}{p}.
\end{eqnarray*} 
Dividing both sides of inequality by $|h|^s$ and letting $h$ tend to $0$, by the a priori assumption $H_{p/2}(Du)\in W^{1,s}_{loc}(\Omega),$ we get
\begin{eqnarray}
\label{prefine}
&&\int_{B_{\rho_1}}|D(H_{p/2}(Du))|^s\, dx \leq \bigg(\frac{CR^{\left(\beta(2-s)-\frac{n(2-s)^2}{s}\right)}}{\nu}+\frac{6\sigma}{\nu}\bigg)^{\frac{s}{2}}\int_{B_{\rho_2}}|D(H_{p/2}(Du))|^s\, dx\cr\cr
&&+\frac{C_\sigma}{(r_2-r_1)^{\frac{ps}{2}}}\bigg(1\!+\!\int_{B_{R}}(|u|^{{\frac{ps}{2}}}+|Du|^{{\frac{ps}{2}}})\, dx\bigg)^{\kappa}\!\!+\!\frac{C_\sigma}{(r_2-r_1)^{\frac{ps}{2}}}\bigg(\int_{B_R}|Du|^p|x|^\beta \, dx\bigg)^\frac{s}{p}.
\end{eqnarray}
Our next aim is to reabsorb the first integral in the right-hand side by the left-hand side. Note that, by our choice of $s$, we have  $\beta(2-s)-\frac{n(2-s)^2}{s}>0$  and hence we may choose $R_0$ small enough to obtain $$R_0^{\left(\beta(2-s)-\frac{n(2-s)^2}{s}\right)}=\frac{1}{4C},$$ hence, for every $ R<R_0$
\begin{equation}
CR^{\left(\beta(2-s)-\frac{n(2-s)^2}{s}\right)}+\frac{6\sigma}{\nu}\leq \frac{1}{4}+\frac{6\sigma}{\nu}.
\end{equation}
Choosing $\sigma=\frac{\nu}{24}$, we have that \begin{equation}
\bigg(CR^{\left(\beta(2-s)-\frac{n(2-s)^2}{s}\right)}+\frac{6\sigma}{\nu}\bigg)^{\frac{s}{2}}\leq \bigg(\frac{1}{2}\bigg)^{\frac{s}{2}}\leq \frac{1}{2}.
\end{equation}
Moreover, noting that our estimate holds for all $\rho_1<r_1<r_2<\rho_2$ choosing $r_1$ and $r_2$ such that $$\frac{1}{r_2-r_1}\cong\frac{1}{\rho_2-\rho_1},$$ estimate \eqref{prefine} can be written as follows 
\begin{eqnarray*}
&&\int_{B_{\rho_1}}|D(H_{p/2}(Du))|^s\, dx \leq \frac{1}{2}\int_{B_{\rho_2}}|D(H_{p/2}(Du))|^s\, dx\cr\cr
&&+\frac{C}{(\rho_2-\rho_1)^{\frac{ps}{2}}}\bigg(1+\int_{B_{R}}(|u|^{{\frac{ps}{2}}}+|Du|^{{\frac{ps}{2}}})\, dx\bigg)^{\kappa}+\frac{C}{(\rho_2-\rho_1)^{\frac{ps}{2}}}\bigg(\int_{B_R}|Du|^p|x|^\beta \, dx\bigg)^\frac{s}{p}. 
\end{eqnarray*}
At this point, we may use the iteration Lemma \ref{lem:Giusti2} with $Z(\rho)=\int_{B_\rho}|D(H_{p/2}(Du))|^s\, dx$, that yields
\begin{equation}
\label{fine}
\int_{B_{R/2}}|D(H_{p/2}(Du))|^s\, dx \leq C\bigg(1+\int_{B_{R}}(|u|^{{\frac{ps}{2}}}+|Du|^{{\frac{ps}{2}}})\, dx\bigg)^{\kappa}+C\bigg(\int_{B_R}|Du|^p|x|^\beta \, dx\bigg)^\frac{s}{p},
\end{equation} which is the conclusion.

\end{proof}

\section{Proof of Theorem \ref{teo3}}
The aim of this section is to conclude the proof of Theorem \ref{teo3} by using a suitable approximation argument. More precisely, let $u\in W^{1,p}_{loc}(\Omega, |x|^\beta dx)$ be a local solution of \eqref{equazione mia}, fix a ball $B_R\Subset\Omega$ and assume, $R< R_0,$ where $R_0$ has been determined in Theorem \ref{teo2}. For $\varepsilon>0$ and $\eta>0$, let $u_{\varepsilon,\eta}\in W^{1,p}(B_R)$ be the unique solution of the following problem
\begin{equation}
    \label{equ regolare}
    \left\{
    \begin{array}{lll}
    \mathrm{div}\bigg((\varepsilon+|x|^\beta)H_{p-1}(Du_{\varepsilon,\eta})\bigg) = |u_\eta|^{q-2}u_\eta f_\varepsilon(x) & \text{in } B_R \\
    u_{\varepsilon,\eta}  = u_\eta & \text{on } \partial B_R
    \end{array}
    \right.
\end{equation}
where
\begin{itemize}
\item $f_\varepsilon(x):=\frac{1}{|x|^\alpha}\ast \psi_\varepsilon$,
\item $u_\eta(x)=u(x)\ast \varrho_\eta$
\end{itemize}
with $\{\psi_\varepsilon\}_{\varepsilon>0}$ and $\{\varrho_\eta\}_{\eta>0}$ families of standard compactly supported $C^\infty$ mollifiers.\\
\begin{proof}[\textbf{Proof of Theorem $1.1$}]

Observe that the right hand side of the equation in \eqref{equ regolare}  is $C^\infty_{loc}(B_R)$ and hence it satisfies the hypotheses for example of \cite{{AmG},{BraCaSan}}, that yields $H_{p/2}(Du_{\varepsilon,\eta})\in W_{loc}^{1,2}(B_R)$. Hence, choosing $s_0<s<2$, where $s_0$ is defined in \eqref{cond s}, we can use Theorem \ref{teo2} with $u_{\varepsilon,\eta}$ in place of $u$ to deduce that 
\begin{align}
\label{approx} \nonumber
\int_{B_r}|D(H_{p/2}(Du_{\varepsilon,\eta}))|^s\, dx&\leq C\left(\int_{{B_{R}}}|Du_{\varepsilon,\eta}|^p|x|^{\beta}\, dx\right)^{\frac{s}{p}}\\
&+C\bigg(1+\int_{B_{R}}(|u_{\varepsilon,\eta}|^{\frac{ps}{2}}+|Du_{\varepsilon,\eta}|^{\frac{ps}{2}})\, dx\bigg)^{\kappa}\\ 
&\leq C\bigg(1+\int_{B_{R}}(|Du_{\varepsilon,\eta}
|^p+|u_{\varepsilon,\eta}|^p)|x|^\beta\, dx\bigg)^{\kappa},
\end{align}
for every $B_r\subset B_\rho\Subset B_R,$, with a constant $C$ independent of $\varepsilon$. 
Note that \begin{eqnarray}
\label{uvar}
    \int_{B_R}|u_{\varepsilon,\eta}|^{(\frac{ps}{2})^*}\, dx&=&\int_{B_R}|u_{\varepsilon,\eta}-u_\eta+u_\eta|^{(\frac{ps}{2})^*}\, dx\cr\cr
    &\leq&\int_{B_R}\bigg(|u_{\varepsilon,\eta}-u_\eta|^{(\frac{ps}{2})^*}+|u_\eta|^{(\frac{ps}{2})^*}\bigg)\, dx\cr\cr
    &\leq& C\bigg(\int_{B_R}|Du_{\varepsilon,\eta}-Du_\eta|^{\frac{ps}{2}}\, dx\bigg)^{\frac{2n}{2n-ps}}+\int_{B_R}|u_\eta|^{(\frac{ps}{2})^*}\, dx\cr\cr
    &\leq& C\bigg(1+\int_{B_R}|Du_{\varepsilon,\eta}|^{\frac{ps}{2}}\, dx+\int_{B_R}(|Du_\eta|^{\frac{ps}{2}}+|u_\eta|^{\frac{ps}{2}})\, dx\bigg)^{\frac{2n}{2n-ps}},
\end{eqnarray} where we use Sobolev-Poincaré inequality and Sobolev embedding theorem.
Using \eqref{uvar} in \eqref{approx} we obtain
\begin{eqnarray}
\label{veraapprox}\nonumber
\int_{B_r}|D(H_{p/2}(Du_{\varepsilon,\eta}))|^s\, dx&\leq&C\left(\int_{{B_{R}}}|Du_{\varepsilon,\eta}|^p|x|^{\beta}\, dx\right)^{\frac{s}{p}}\cr\cr
&+&C\bigg(1+\int_{B_R}|Du_{\varepsilon,\eta}|^{\frac{ps}{2}}\, dx+\int_{B_R}(|Du_\eta|^{\frac{ps}{2}}+|u_\eta|^{\frac{ps}{2}})\, dx\bigg)^{\kappa}\\ 
&\leq& C\bigg(1+\int_{B_{R}}(|Du_{\varepsilon,\eta}
|^p+|Du_\eta|^p+|u_\eta|^p)|x|^\beta\, dx\bigg)^{\kappa}
\end{eqnarray}
Since $u_{\varepsilon,\eta}$ is a weak solution of \eqref{equ regolare}, choosing $\varphi=u_{\varepsilon,\eta}-u_\eta$ as test function, we obtain \noindent
\begin{align*}
0&=\int_{B_R} \bigl\langle (\varepsilon+|x|^\beta)H_{p-1}(Du_{\varepsilon,\eta}),Du_{\varepsilon,\eta}-Du_\eta\bigr\rangle \, dx-\int_{B_R} |u_\eta|^{q-2}u_\eta f_\varepsilon(x)(u_{\varepsilon,\eta}-u_\eta)\textbf{ } \, dx\\
   &=\int_{B_R}\bigl\langle (\varepsilon+|x|^\beta)H_{p-1}(Du_{\varepsilon,\eta}),Du_{\varepsilon,\eta}\bigr\rangle \, dx
-\int_{B_R} \bigl\langle (\varepsilon+|x|^\beta)H_{p-1}(Du_{\varepsilon,\eta}),Du_\eta\bigr\rangle \, dx\\
&-\int_{B_R} (|u_\eta|^{q-2}u_\eta)f_\varepsilon(x)(u_{\varepsilon,\eta}-u_\eta)\textbf{ } \, dx\\
&=:J_1-J_2-J_3,
\end{align*}
that yields 
\begin{equation}
\label{dis:app}
    J_1\leq |J_2|+|J_3|.
\end{equation}  
By Lemma \ref{lem:Brasco}, we deduce that
\begin{align}
    \label{lhs}
    J_1&\geq \nu\int_{B_R} {|x|^\beta}|H_{p/2}(Du_{\varepsilon,\eta})|^2\, dx+\nu\varepsilon\int_{B_R}|H_{p/2}(Du_{\varepsilon,\eta})|^2\, dx.
\end{align}
In order to estimate $|J_2|$, we use Young's inequality and the definition of $H_\delta(\xi)$ as follows
\begin{align}
    \label{rhs1}\nonumber
    |J_2|&\leq\int_{B_R} {|x|^\beta}|H_{p-1}(Du_{\varepsilon,\eta})||Du_\eta| \, dx+\varepsilon\int_{B_R}|H_{p-1}(Du_{\varepsilon,\eta})||Du_\eta|\, dx\\ \nonumber
    &\leq \sigma\int_{B_R} {|x|^\beta}|H_{p/2}(Du_{\varepsilon,\eta})|^2\, dx+C_\sigma \int_{B_R} |x|^\beta|Du_\eta|^p\, dx+C_\nu\varepsilon \int_{B_R} |Du_\eta|^p\, dx\\ \nonumber &+\frac{\nu\varepsilon}{2}\int_{B_R}|H_{p/2}(Du_{\varepsilon,\eta})|^{\frac{p}{p-1}}\, dx\\ \nonumber
    &\leq \sigma\int_{B_R} {|x|^\beta}|H_{p/2}(Du_{\varepsilon,\eta})|^2\, dx+C_\sigma \int_{B_R} |x|^\beta|Du_\eta|^p\, dx+C_\nu\varepsilon\int_{B_R} |Du_\eta|^p\, dx\\
    &+\frac{\nu\varepsilon}{2}\int_{B_R}|H_{p/2}(Du_{\varepsilon,\eta})|^{2}\, dx.
    \end{align}
Now we estimate $|J_3|$ using H\"older's inequality with exponents $$\bigg(\frac{ps}{2}\bigg)^*, \bigg(\bigg(\frac{ps}{2}\bigg)^*\bigg)',\frac{nps}{nps-q(2n-ps)}$$ and Sobolev-Poincaré theorem, as follows 
\begin{align*}
|J_3|&\leq \bigg(\int_{B_R}|u_{\varepsilon,\eta}-u_\eta|^{(\frac{ps}{2})^*}\, dx\bigg)^{\frac{1}{(\frac{ps}{2})^*}}\bigg(\int_{B_R}f_\varepsilon(x)^{\frac{nps}{nps-q(2n-ps)}}dx\bigg)^\frac{nps-q(2n-ps)}{nps}\bigg(\int_{B_R}|u_\eta|^{(\frac{ps}{2})^*}\, dx\bigg)^{\frac{q-1}{(\frac{ps}{2})^*}}\\\nonumber
&\leq \bigg(\int_{B_R}|Du_{\varepsilon,\eta}-Du_\eta|^{\frac{ps}{2}}\, dx\bigg)^{\frac{2}{ps}}\bigg(\int_{B_R}f_\varepsilon(x)^{\frac{nps}{nps-q(2n-ps)}}dx\bigg)^\frac{nps-q(2n-ps)}{nps}\bigg(\int_{B_R}|u_\eta|^{(\frac{ps}{2})^*}\, dx\bigg)^{\frac{q-1}{(\frac{ps}{2})^*}},
\end{align*}
where we note that the use of H\"older's inequality was legitimate since $s_0<s.$ We can estimate $|J_3|$ further using Sobolev embedding theorem
\begin{align*}
|J_3|&\leq \bigg(\int_{B_R}|Du_{\varepsilon,\eta}-Du_\eta|^{\frac{ps}{2}}\, dx\bigg)^{\frac{2}{ps}}\bigg(\int_{B_R}f_\varepsilon(x)^{\frac{nps}{nps-q(2n-ps)}}dx\bigg)^\frac{nps-q(2n-ps)}{nps}\\\nonumber
&\cdot\bigg(\int_{B_{R}}|Du_\eta|^{\frac{ps}{2}}\, dx+\frac{1}{R^{\frac{ps}{2}}}\int_{B_{R}}|u_\eta|^{\frac{ps}{2}}\, dx+R^{n-\frac{ps}{2}}\bigg)^{\frac{q-1}{\frac{ps}{2}}}\\\nonumber
&\leq\bigg(\int_{|Du_{\varepsilon,\eta}|>1}((|Du_{\varepsilon,\eta}|-1)_+^{\frac{ps}{2}}+1)\, dx\bigg)^{\frac{2}{ps}}\bigg(\int_{B_R}f_\varepsilon(x)^{\frac{nps}{nps-q(2n-ps)}}dx\bigg)^\frac{nps-q(2n-ps)}{nps} \\ \nonumber
&\cdot\bigg(\int_{B_{R}}|Du_\eta|^{\frac{ps}{2}}\, dx+\frac{1}{R^{\frac{ps}{2}}}\int_{B_{R}}|u_\eta|^{\frac{ps}{2}}\, dx+R^{n-\frac{ps}{2}}\bigg)^{\frac{q-1}{\frac{ps}{2}}}\\\nonumber
&+\bigg(\int_{B_R\cap {|Du_{\varepsilon,\eta}|<1}}|Du_{\varepsilon,\eta}|^{\frac{ps}{2}}\, dx\bigg)^{\frac{2}{ps}}\bigg(\int_{B_R}f_\varepsilon(x)^{\frac{nps}{nps-q(2n-ps)}}dx\bigg)^\frac{nps-q(2n-ps)}{nps}\\ \nonumber
&\cdot\bigg(\int_{B_{R}}|Du_\eta|^{\frac{ps}{2}}\, dx+\frac{1}{R^{\frac{ps}{2}}}\int_{B_{R}}|u_\eta|^{\frac{ps}{2}}\, dx+R^{n-\frac{ps}{2}}\bigg)^{\frac{q-1}{\frac{ps}{2}}} \\ \nonumber
&+\bigg(\int_{B_R}|Du_\eta|^{\frac{ps}{2}}\, dx\bigg)^{\frac{2}{ps}}\bigg(\int_{B_R}f_\varepsilon(x)^{\frac{nps}{nps-q(2n-ps)}}dx\bigg)^\frac{nps-q(2n-ps)}{nps}\\ \nonumber
&\cdot\bigg(\int_{B_{R}}|Du_\eta|^{\frac{ps}{2}}\, dx+\frac{1}{R^{\frac{ps}{2}}}\int_{B_{R}}|u_\eta|^{\frac{ps}{2}}\, dx+R^{n-\frac{ps}{2}}\bigg)^{\frac{q-1}{\frac{ps}{2}}} \\ \nonumber
&\leq C\bigg(\int_{B_R}(1+|H_{p/2}(Du_{\varepsilon,\eta})|^s)\, dx\bigg)^{\frac{2}{ps}}\bigg(\int_{B_R}f_\varepsilon(x)^{\frac{nps}{nps-q(2n-ps)}}dx\bigg)^\frac{nps-q(2n-ps)}{nps} \\ \nonumber
&\cdot\bigg(\int_{B_{R}}|Du_\eta|^{\frac{ps}{2}}\, dx+\frac{1}{R^{\frac{ps}{2}}}\int_{B_{R}}|u_\eta|^{\frac{ps}{2}}\, dx+R^{n-\frac{ps}{2}}\bigg)^{\frac{q-1}{\frac{ps}{2}}}\\
&+\bigg(\int_{B_{R}}|Du_\eta|^{\frac{ps}{2}}\, dx+\frac{1}{R^{\frac{ps}{2}}}\int_{B_{R}}|u_\eta|^{\frac{ps}{2}}\, dx+R^{n-\frac{ps}{2}}\bigg)^{\frac{2q}{ps}}\bigg(\int_{B_R}f_\varepsilon(x)^{\frac{nps}{nps-q(2n-ps)}}dx\bigg)^\frac{nps-q(2n-ps)}{nps}.
\end{align*}
Arguing as in \eqref{dis:s} and using Young's inequality, we get
\begin{align}
\label{rhs2}\nonumber
|J_3|&\leq C\bigg(\int_{B_R}|x|^\beta|H_{p/2}(Du_{\varepsilon,\eta})|^2\, dx\bigg)^{\frac{1}{p}}\bigg(\int_{B_R}f_\varepsilon(x)^{\frac{nps}{nps-q(2n-ps)}}dx\bigg)^\frac{nps-q(2n-ps)}{nps} \\ \nonumber
&\cdot\bigg(1+\int_{B_{R}}|Du_\eta|^{\frac{ps}{2}}\, dx+\int_{B_{R}}|u_\eta|^{\frac{ps}{2}}\, dx\bigg)^{\frac{q-1}{\frac{ps}{2}}}\\ \nonumber
&+C\bigg(1+\int_{B_{R}}|Du_\eta|^{\frac{ps}{2}}\, dx+\int_{B_{R}}|u_\eta|^{\frac{ps}{2}}\, dx\bigg)^{\frac{2q}{ps}}\bigg(\int_{B_R}f_\varepsilon(x)^{\frac{nps}{nps-q(2n-ps)}}dx\bigg)^\frac{nps-q(2n-ps)}{nps}\\\nonumber
&\leq \sigma\int_{B_R}|x|^\beta|H_{p/2}(Du_{\varepsilon,\eta})|^2\, dx\\
&+C\bigg(1+\int_{B_{R}}|Du_\eta|^{\frac{ps}{2}}\, dx+\int_{B_{R}}|u_\eta|^{\frac{ps}{2}}\, dx\bigg)^{\frac{2(q-1)}{s(p-1)}}\bigg(\int_{B_R}f_\varepsilon(x)^{\frac{nps}{nps-q(2n-ps)}}dx\bigg)^\frac{nps-q(2n-ps)}{nps},
\end{align} where $C=C(R).$
We can estimate $|J_3|$ further by noting that 
\begin{align}
\label{dis:xalpha}
||f_\varepsilon(x)||_{L^{\frac{nps}{nps-q(2n-ps)}}(B_R)}&=&\bigg|\bigg|\bigg(\frac{1}{|x|^\alpha}\bigg)_\varepsilon\bigg|\bigg|_{L^{\frac{nps}{nps-q(2n-ps)}}(B_R)}\leq C\bigg|\bigg|\frac{1}{|x|^\alpha}\bigg|\bigg|_{L^{\frac{nps}{nps-q(2n-ps)}}(B_R)}\leq C(p,q,n,R), \ \ \   
\end{align} 
where we used the standard properties of the mollifiers.\\
Using \eqref{dis:xalpha} in \eqref{rhs2}, we finally obtain
\begin{eqnarray} \label{rhs21}
   |J_3|&\leq& \sigma\int_{B_R}|x|^\beta|H_{p/2}(Du_{\varepsilon,\eta})|^2\, dx+C_\sigma\bigg(1+\int_{B_{R}}(|Du_\eta|^{\frac{ps}{2}}+|u_\eta|^{\frac{ps}{2}})\, dx\bigg)^{\frac{2(q-1)}{s(p-1)}}. 
\end{eqnarray}

Inserting \eqref{lhs}, \eqref{rhs1} and \eqref{rhs21}  in \eqref{dis:app} we obtain
\begin{align}
\label{m}
\nonumber
   &\nu\int_{B_R}{|x|^\beta}|H_{p/2}(Du_{\varepsilon,\eta})|^2\, dx+\nu\varepsilon\int_{B_R}|H_{p/2}(Du_{\varepsilon,\eta})|^2\, dx\leq 3\sigma\int_{B_R} |x|^\beta|H_{p/2}(Du_{\varepsilon,\eta})|^2\, dx\\ \nonumber
   &+\frac{\nu\varepsilon}{2}\int_{B_R}|H_{p/2}(Du_{\varepsilon,\eta})|^2\, dx+C_\sigma \int_{B_R} |x|^\beta|Du_\eta|^p\, dx+C_\nu\varepsilon\int_{B_R} |Du_\eta|^p\, dx\\ 
   &+C_\sigma\bigg(1+\int_{B_{R}}(|Du_\eta|^{\frac{ps}{2}}+|u_\eta|^{\frac{ps}{2}})\, dx\bigg)^{\frac{2(q-1)}{s(p-1)}}.
\end{align}
Choosing $\sigma=\frac{\nu}{6}$ we can reabsorb the first and the second integral in the right-hand side by the left-hand side, and reabsorbing also the second integral in \eqref{m}, we get
\begin{align*}
 \frac{\nu}{2}\int_{B_R}|x|^\beta |H_{p/2}(Du_{\varepsilon,\eta})|^2\, dx &+\frac{\nu\varepsilon}{2}\int_{B_R}|H_{p/2}(Du_{\varepsilon,\eta})|^2\, dx\leq C \int_{B_R} |x|^\beta|Du_\eta|^p\, dx\\ &+C_\nu\varepsilon\int_{B_R} |Du_\eta|^p\, dx+C\bigg(1+\int_{B_{R}}(|Du_\eta|^{\frac{ps}{2}}+|u_\eta|^{\frac{ps}{2}})\, dx\bigg)^{\frac{2(q-1)}{s(p-1)}},
\end{align*}
that implies
\begin{eqnarray}
\label{cc}\nonumber
 \int_{B_{R}}|x|^\beta|Du_{\varepsilon,\eta}|^p\, dx&\leq& \int_{|Du_{\varepsilon,\eta}|<1}|x|^\beta|Du_{\varepsilon,\eta}|^p\, dx+\int_{|Du_{\varepsilon,\eta}|>1}|x|^\beta\bigg((|Du_{\varepsilon,\eta}|-1)^p_++1\bigg)\, dx\\\nonumber
 &\leq& C\int_{B_R}|x|^\beta dx+\int_{B_R}|x|^\beta(1+|H_{p/2}(Du_\eta)|^2)dx\\
 &\leq& C(R,||Du_\eta||_{L^p(B_R,|x|^\beta dx)})
\end{eqnarray}and
\begin{eqnarray}
\label{c}
\int_{B_{R}}|Du_{\varepsilon,\eta}|^{\frac{ps}{2}}\, dx&\le&\left(\int_{B_{R}}|x|^\beta|Du_{\varepsilon,\eta}|^p\, dx\right)^{\frac{s}{2}} \bigg(\int_{B_R}|x|^{\frac{-\beta s}{2-s}}\bigg)^{\frac{2-s}{2}} \cr\cr
 &\le& C(R,||Du_\eta||_{L^p(B_R,|x|^\beta dx)})
\end{eqnarray}
with constants both independent of $\varepsilon$.
Using \eqref{cc} and \eqref{c} in the right-hand side of \eqref{veraapprox}, we get
\begin{align}
\label{approx3}\nonumber
\int_{B_r}|DH_{p/2}(Du_{\varepsilon,\eta})|^s\, dx &\leq C \int_{B_R} |x|^\beta|Du_\eta|^p\, dx+C_\nu\varepsilon\int_{B_R} |Du_\eta|^p\, dx\\&+C\bigg(1+\int_{B_{R}}(|Du_\eta|^{\frac{ps}{2}}+|u_\eta|^{\frac{ps}{2}})\, dx\bigg)^{\frac{2(q-1)}{s(p-1)}}.
\end{align}
Notice that, since we aim to take the limit as $\varepsilon \to 0$, we can suppose $\varepsilon \le 1$, so that the right-hand side is bounded independently of $\varepsilon$.\\
By \eqref{c} and \eqref{approx3}, as $\varepsilon\to 0$, we have 
\begin{align}
\label{converge} 
u_{\varepsilon,\eta}&\rightharpoonup v_\eta \text{ in }W^{1,{\frac{ps}{2}}}\\ 
H_{p/2}(Du_{\varepsilon,\eta}) &\rightharpoonup w_\eta \text{ in } W^{1,s}_{loc}(B_R)
\end{align} 
which yields $$H_{p/2}(Du_{\varepsilon,\eta}) \longrightarrow w_\eta \text{ in } L^{\varsigma}_{loc}(B_R) \ \ \forall \varsigma<s^*,$$
and therefore also a.e. up to a subsequence. By the continuity of the function $H_{p/2}(\xi)$ and the uniqueness of the weak limit, it holds $$w_\eta=H_{p/2}(Dv_\eta).$$
Now note that passing to the limit as $\varepsilon\to 0$ in \eqref{approx3} and applying Fatou's Lemma, we obtain
\begin{equation}
    \label{approx4}
 \int_{B_r}|DH_{p/2}(Dv_{\eta})|^s\, dx \leq C \int_{B_R} |x|^\beta|Du_\eta|^p\, dx+C\bigg(1+\int_{B_{R}}(|Du_\eta|^{\frac{ps}{2}}+|u_\eta|^{\frac{ps}{2}})\, dx\bigg)^{\frac{2(q-1)}{s(p-1)}}.   
\end{equation}
Our next aim is to prove that the function $v_\eta\in W^{1,p}(B_R, |x|^\beta dx)$ is a solution of \eqref{equazione mia}, i.e. in weak formulation, that the following identity
\begin{equation}
    \label{verifica2}
    \int_{B_R}\bigg(|x|^\beta\langle H_{p-1}(Dv_\eta),D\varphi\rangle \ \, dx=\int_{B_R}\frac{1}{|x|^\alpha}|u_\eta|^{q-2}u_\eta \ \varphi \ \, dx, 
\end{equation}
holds for all $\varphi\in W^{1,p}_0(B_R,|x|^\beta dx).$
In order to prove \eqref{verifica2}, we observe that  
\begin{align}\label{verifica3}
\nonumber
&\int_{B_R}|x|^\beta\langle  H_{p-1}(Dv_\eta), D\varphi\rangle \, dx=\int_{B_R}|x|^\beta\langle H_{p-1}(Dv_\eta),D\varphi\rangle \ \, dx \\ \nonumber
&+\int_{B_R}(\varepsilon+|x|^\beta)\langle H_{p-1}(Du_{\varepsilon,\eta}), D\varphi \rangle \ \, dx-\int_{B_R}(\varepsilon+|x|^\beta)\langle H_{p-1}(Du_{\varepsilon,\eta}), D\varphi\rangle \ \, dx\\ \nonumber
&=\int_{B_R}|x|^\beta\langle H_{p-1}(Dv_\eta),D\varphi\rangle \ \, dx+\int_{B_R}|u_\eta|^{q-2}u_\eta f_\varepsilon(x)\varphi \ \, dx\\\nonumber
&-\int_{B_R}(\varepsilon+|x|^\beta)\langle H_{p-1}(Du_{\varepsilon,\eta}),D\varphi\rangle \ \, dx\\\nonumber
&=\int_{B_R}|x|^\beta\langle H_{p-1}(Dv_\eta)-H_{p-1}(Du_{\varepsilon,\eta}), D\varphi\rangle \, dx-\varepsilon\int_{B_R}\langle H_{p-1}(Du_{\varepsilon,\eta}),D\varphi\rangle \ \, dx\\ \nonumber
&+\int_{B_R}\bigg(|u_\eta|^{q-2}u_\eta f_\varepsilon(x)-\frac{1}{|x|^\alpha}|u_\eta|^{q-2}u_\eta\bigg)\varphi \ \, dx+\int_{B_R}\frac{1}{|x|^\alpha}|u_\eta|^{q-2}u_\eta\varphi \, dx\\ \nonumber
&= \int_{B_R}|x|^\beta\langle H_{p-1}(Dv_\eta)-H_{p-1}(Du_{\varepsilon,\eta}), D\varphi\rangle \, dx+\int_{B_R}\bigg(\frac{1}{|x|^\alpha}-\bigg(\frac{1}{|x|^\alpha}\bigg)_\varepsilon\bigg)|u_\eta|^{q-2}u_\eta\varphi\\ \nonumber
&-\varepsilon\int_{B_R}\langle H_{p-1}(Du_{\varepsilon,\eta}),D\varphi\rangle \ \, dx+\int_{B_R}\frac{1}{|x|^\alpha}|u_\eta|^{q-2}u_\eta \ \varphi \, dx\\
&=:J_\varepsilon+JJ_\varepsilon-\varepsilon\int_{B_R}\langle H_{p-1}(Du_{\varepsilon,\eta}),D\varphi\rangle \ \, dx+\int_{B_R}\frac{1}{|x|^\alpha}|u_\eta|^{q-2}u_\eta \ \varphi \, dx,
\end{align}
where, in the second identity we used that $u_{\varepsilon,\eta}$ solves problem  \eqref{equ regolare}. Note that using Young's and H\"older's inequalities, we get
\begin{align}
\label{dis:H}
\nonumber
\varepsilon\int_{B_R}\langle H_{p-1}(Du_{\varepsilon,\eta}),D\varphi\rangle \ \, dx&\leq\frac{\varepsilon}{2}\int_{B_R} |x|^\beta|D\varphi|^p\, dx+\frac{\varepsilon}{2}\int_{B_R}|x|^{\frac{-\beta}{p-1}}|H_{p-1}(Du_{\varepsilon,\eta})|^{\frac{p}{p-1}}\, dx\\ \nonumber
&\leq \frac{\varepsilon}{2}||D\varphi||_{L^p(B_R,|x|^\beta dx)}+\frac{\varepsilon}{2}\int_{supp \ D\varphi}|x|^{\frac{-\beta}{p-1}}|H_{p-1}(Du_{\varepsilon,\eta})|^{\frac{p}{p-1}}\, dx\\\nonumber
&= \frac{\varepsilon}{2}||D\varphi||_{L^p(B_R,|x|^\beta dx)}+\frac{\varepsilon}{2}
\int_{supp \ D\varphi}|x|^{\frac{-\beta}{p-1}}|H_{p/2}(Du_{\varepsilon,\eta})|^2\, dx\\ \nonumber
&\leq \frac{\varepsilon}{2}||D\varphi||_{L^p(B_R,|x|^\beta dx)}\\&\nonumber+\frac{\varepsilon}{2}\bigg(\int_{supp \ D\varphi}|H_{p/2}(Du_{\varepsilon,\eta})|^{2\tilde{s}}\, dx\bigg)^{\frac{1}{\tilde{s}}}\bigg(\int_{B_R}|x|^{\frac{-\beta}{p-1}\frac{\tilde{s}}{\tilde{s}-1}}dx\bigg)^{\frac{\tilde{s}-1}{\tilde{s}}}\\
&\leq \varepsilon C(R,||D\varphi||_{L^p(B_R,|x|^\beta dx)},\eta)
\end{align} where we chose $\frac{n(p-1)}{n(p-1)-\beta}<\tilde{s}<s^*
,$ we used \eqref{approx3} and $C$ is independent of $\varepsilon.$
We notice that the right hand side \eqref{dis:H} goes to $0$ as $\varepsilon \rightarrow 0.$ 
Our aim is to prove that $J_\varepsilon$ and $JJ_\varepsilon$ tend to $0$ as well, as $\varepsilon$ goes to $0.$ Indeed, we have
\begin{align}
  \label{j1} \nonumber
  |J_\varepsilon|&\leq ||D\varphi||_{L^p(B_R,|x|^\beta dx)}\bigg(\int_{supp \ D\varphi}|x|^\beta|H_{p-1}(Dv_\eta)-H_{p-1}(Du_{\varepsilon,\eta})|^\frac{p}{p-1}\, dx\bigg)^{\frac{p-1}{p}}\\ \nonumber
  &\leq C\bigg(\int_{supp \ D\varphi}|x|^\beta|H_{p/2}(Dv_\eta)-H_{p/2}(Du_{\varepsilon,\eta})|^{\frac{p}{p-1}}\bigg((|Dv_\eta|-1)_++(|Du_{\varepsilon,\eta}|-1)_+\bigg)^{\frac{p}{p-1}\frac{p-2}{2}}\, dx\bigg)^{\frac{p-1}{p}}\\ \nonumber
&\leq C\bigg(\int_{supp \ D\varphi}|H_{p/2}(Dv_\eta)-H_{p/2}(Du_{\varepsilon,\eta})|^2\, dx\bigg)^{\frac{p}{2(p-1)}}\\ \nonumber
&\cdot\bigg(\int_{supp \ D\varphi}|x|^{\beta}\bigg((|Dv_\eta|-1)_++(|Du_{\varepsilon,\eta}|-1)_+\bigg)^{\frac{2(p-1)}{p-2}}\, dx\bigg)^{\frac{p-2}{2p}}\\
&\leq C\bigg(\int_{supp \ D\varphi}|H_{p/2}(Dv_\eta)-H_{p/2}(Du_{\varepsilon,\eta})|^2\, dx\bigg)^{\frac{p}{2(p-1)}}\bigg(\int_{B_R}|x|^\beta(|Dv_\eta|^{p}+|Du_{\varepsilon,\eta}|^{p})\, dx\bigg)^{\frac{p-2}{2p}},
\end{align}where we used Lemma \ref{eq:BraAmb} with $\epsilon=p-1$ and $\theta=p/2$ and H\"older's inequality. Note that since $s>s_0$ in particular we may choose $s>\frac{2n}{n+2}$ so that $2<s^*,$ and so $H_{p/2}(Du_{\varepsilon,\eta}) \to H_{p/2}(Dv_\eta)$ strongly in $L^2_{loc}(B_R),$ and since $\int_{B_R}|Du_{\varepsilon,\eta}|^{\frac{ps}{2}}\, dx$ is bounded by \eqref{c}, then we conclude that $J_\varepsilon \to 0$.\\
For the estimate of $JJ_\varepsilon$ we have to study again the case $1<q\leq 2$ and $q>2$ apart. For $q>2$, since $s>s_0$ by \eqref{cond s} we have $$s>\frac{2nq}{p(n+q-\alpha)}$$ and so we can estimate $JJ_\varepsilon$ by using H\"older's with exponents $$ \bigg(\frac{ps}{2}\bigg)^*, \frac{(\frac{ps}{2})^*}{q-2},\frac{nps}{nps-q(2n-ps)}$$and Sobolev inequalities, getting
\begin{align}
\label{jj} \nonumber
|JJ_\varepsilon|&\leq ||\varphi||_{L^{(\frac{ps}{2})^*}(B_R,|x|^\beta dx)} \bigg(\int_{B_R}\bigg|\frac{1}{|x|^\alpha}-\bigg(\frac{1}{|x|^\alpha}\bigg)_\varepsilon\bigg|^{\frac{nps}{nps-q(2n-ps)}}\, dx\bigg)^{\frac{nps-q(2n-ps)}{nps}}\\ 
&\cdot \bigg(\int_{B_{R}}(1+|Du_\eta|^{\frac{ps}{2}}+|u_\eta|^{\frac{ps}{2}})\, dx\bigg)^{\kappa}.
\end{align}
Since $\bigg(\frac{1}{|x|^\alpha}\bigg)_\varepsilon \to \frac{1}{|x|^\alpha}$ strongly in $L^{\frac{nps}{nps-q(2n-ps)}}(B_R),$ we can conclude that $JJ_\varepsilon \to 0$ as $\varepsilon \to 0.$ \\

The case $1<q\leq 2$ follows using $$|u_\eta|^{q-1} \leq 1+|u_\eta|,$$ H\"older's with exponents $$\bigg(\frac{ps}{2}\bigg)^*, \frac{(\frac{ps}{2})^*}{(\frac{ps}{2})^* - 2},\bigg(\frac{ps}{2}\bigg)^*$$
and Sobolev inequalities, obtaining
\begin{align*}
\int_{B_R} \left| \frac{1}{|x|^\alpha} - \left(\frac{1}{|x|^\alpha}\right)_\varepsilon \right| |u_\eta|^{q-1} |\varphi| \, dx 
&\leq C \|\varphi\|_{L^{(\frac{ps}{2})^*}(B_R,|x|^\beta dx)} \bigg(\int_{B_R} \left| \frac{1}{|x|^\alpha} - \left(\frac{1}{|x|^\alpha}\right)_\varepsilon \right|^{\frac{(\frac{ps}{2})^*}{(\frac{ps}{2})^* - 2}} dx\bigg)^{\frac{{(\frac{ps}{2})^* - 2}}{(\frac{ps}{2})^*}}\\ &\cdot\left( \int_{B_R} \left(1+|Du_\eta|^{\frac{ps}{2}} + |u_\eta|^{\frac{ps}{2}}\right) \, dx \right)^{\kappa}.
\end{align*} We note again that since $\bigg(\frac{1}{|x|^\alpha}\bigg)_\varepsilon \to \frac{1}{|x|^\alpha}$ strongly in $L^{\frac{nps}{nps-q(2n-ps)}}(B_R),$ we can conclude that $JJ_\varepsilon \to 0$ as $\varepsilon \to 0.$ 
Our next aim is to prove that $H_{p/2}(Du_\eta)=H_{p/2}(Dv_\eta).$ Since $u_\eta$ and $v_\eta$ solve \eqref{equ regolare} we have
\begin{equation*}
\int_{B_R}|x|^\beta\langle H_{p-1}(Du_\eta)-H_{p-1}(Dv_\eta), D\varphi\rangle \, dx=0
\end{equation*}
and, choosing $\varphi=u_\eta-v_\eta$ as a test function, we obtain
\begin{align}
    \label{fineapprox}
    0\leq \int_{B_R}|x|^\beta|H_{p/2}(Du_\eta)-H_{p/2}(Dv_\eta)|^2\, dx &\leq \int_{B_R}|x|^\beta\langle H_{p-1}(Du_\eta)-H_{p-1}(Dv_\eta), Du_\eta-Dv_\eta\rangle \, dx=0,
\end{align}
where we used the degenerate ellipticity of $H_{p/2}(\xi)$ given by Lemma \ref{lem:Brasco}.\\ Hence \eqref{fineapprox} proves that $H_{p/2}(Du_\eta)=H_{p/2}(Dv_\eta)$ and consequently, \eqref{approx4} holds with $u_\eta$ in place of $v_\eta$, i.e.
\begin{equation}
    \label{approx5}
 \int_{B_r}|DH_{p/2}(Du_{\eta})|^s\, dx \leq C \int_{B_R} |x|^\beta|Du_\eta|^p\, dx+C\bigg(1+\int_{B_{R}}(|Du_\eta|^{\frac{ps}{2}}+|u_\eta|^{\frac{ps}{2}})\, dx\bigg)^{\frac{2(q-1)}{s(p-1)}}.   
\end{equation}
Since $u_\eta\to u$ in $W^{1,p}(B_R, |x|^\beta dx)$, passing to the limit as $\eta\to 0$ in \eqref{approx5}, we obtain the desired estimate.

\end{proof}

\smallskip
\par\noindent
{\bf Funding}. This research was partly funded by:\\
\begin{itemize}
    \item GNAMPA of the Italian INdAM-National Institute of High Mathematics (grant number not available);
    \item GNAMPA Project 2026, grant number  E53C25002010001, ``Esistenza e regolarità per soluzioni di equazioni ellittiche e paraboliche anisotrope''
\end{itemize}

\bigskip
\par\noindent
{\bf Conflict of Interest}. The authors declare that they have no conflict of interest.

\end{document}